\documentclass[11pt,a4]{amsart}

\usepackage{mathrsfs}
\usepackage[colorlinks,citecolor=red,pagebackref,hypertexnames=false, breaklinks]{hyperref}

\usepackage{amsmath, color}
\usepackage{enumerate}
\usepackage{amssymb}
\usepackage{amsfonts}

\usepackage{amsthm}
\usepackage{mathrsfs}
\usepackage{fullpage}
\usepackage{setspace}
\usepackage{verbatim}
\usepackage{float}
\usepackage{titlepic}
\usepackage{url}
\usepackage{afterpage}
\usepackage{enumerate}
\usepackage{graphicx}
\usepackage{color}
\usepackage{epstopdf}
\usepackage{mathtools}
\usepackage{hyperref}

\newcommand{\disappear}[1]

\newtheorem{theorem}{Theorem}[section]
\newtheorem{proposition}[theorem]{Proposition}

\newtheorem{lemma}[theorem]{Lemma}

\theoremstyle{remark} \newtheorem{remark}[theorem]{Remark}
\newtheorem{definition}[theorem]{Definition}

\newcommand\C{\mathbb{C}}
\newcommand\N{\mathbb{N}}
\newcommand\R{\mathbb{R}}

\newcommand\M{\mathcal{M}}
\newcommand\E{\mathcal{E}}
\newcommand\V{\Phi}
\newcommand\Z{\mathbb{Z}}

\def\d{{d}}

\newcommand{\supp}{\operatorname{supp}}
\newcommand{\vol}{\operatorname{vol}}
\newcommand\Id{\operatorname{Id}}

\newcommand\Op{\operatorname{Op}}
\newcommand\ilg{\operatorname{ilg}}
\newcommand\dom{\operatorname{dom}}
\newcommand\HS{\mathrm{HS}}

\newcommand{\tG}{\widetilde{G}}
\newcommand{\tE}{\widetilde{E}}
\newcommand{\ttE}{\widetilde{E}'}
\newcommand{\tttE}{\widetilde{E}''}

\newcommand{\dkappa}{\, d\kappa}

\renewcommand{\d}{\partial}
\newcommand{\dx}{\, dx}
\newcommand{\dy}{\, dy}
\newcommand{\dz}{\, dz}
\newcommand{\dr}{\, dr}

\newcommand{\dk}{\, dk}

\begin{document}

\title[Riesz transforms on a class of non-doubling manifolds II]{Riesz transforms on a class of non-doubling manifolds II}
\author{Andrew Hassell, Daniel Nix and Adam Sikora}
\address{Andrew Hassell, Mathematical Sciences Institute, Australian National University,
ACT 2601 Australia
}
\email{Andrew.Hassell@anu.edu.au}
\address{Daniel Nix, Mathematical Sciences Institute, Australian National University,
ACT 2601 Australia
}
\email{Daniel.Nix@anu.edu.au}
\address {Adam Sikora, Department of Mathematics and Statistics, Macquarie University, NSW 2109, Australia}
\email{Adam.Sikora@mq.edu.au}

\subjclass{42B20 (primary), 47F05, 58J05 (secondary).}
\keywords{Riesz transform, heat kernel bounds, resolvent estimates, non-doubling spaces, connected sums}

\begin{abstract}
We consider a class of manifolds $\mathcal{M}$ obtained by taking the connected sum of a finite number of $N$-dimensional Riemannian manifolds of the form $(\mathbb{R}^{n_i}, \delta) \times (\mathcal{M}_i, g)$, 
where $\mathcal{M}_i$ is a compact manifold, with the product metric. The case of greatest interest is when the Euclidean dimensions $n_i$ are not all equal. This means that
the ends have different `asymptotic dimension', and implies that the Riemannian manifold $\mathcal{M}$ is not a doubling space.

In the first paper in this series, by the first and third authors, we considered the case where each $n_i$ is least $3$. In the present paper, we assume that one of the $n_i$ is equal to $2$, which is a special and particularly interesting case. Our approach is to construct the low energy resolvent and determine the asymptotics of the resolvent kernel as the energy tends to zero. We show that the resolvent kernel $(\Delta + k^2)^{-1}$ on $\M$ has an expansion in powers of $1/\log (1/k)$ as $k \to 0$, which is significantly different from the case where all $n_i$ are at least 3, in which case the expansion is in powers of $k$.  We express the Riesz transform in terms of the resolvent to show that it is bounded on $L^p(\M)$ for $1 < p \leq 2$, and unbounded for all $p > 2$. 
%
\end{abstract}

\maketitle

\section{Introduction}\label{sone}
In this paper we continue (with an additional author) the study of the resolvent at low energy and the Riesz transform on connected sums of manifolds begun in \cite{HS2}, which it itself an outgrowth of the earlier paper \cite{CCH}. The classical Riesz transform on $\R^n$ is the vector Fourier multiplier with symbol $\xi/|\xi|$. Having a bounded symbol this operator is clearly bounded on $L^2(\R^n)$, and celebrated work of Riesz \cite{Ri} (in one dimension) and Calder\'on and Zygmund \cite{CZ} (in higher dimensions) showed that it is bounded on $L^p(\R^n)$ for $1 < p < \infty$, as well as of weak type $(1,1)$. Strichartz \cite{Strichartz} introduced a geometric generalization of this question. Given a complete Riemannian manifold $(M, g)$, with Laplace operator $\Delta_g$ (with the sign convention that it is positive as an operator), he asked whether the operator $T = \nabla \Delta_g^{-1/2}$, which coincides with the classical Riesz transform when $(M,g)$ is flat Euclidean space, is bounded on $L^p(M, dg)$. (Recall that $M$ complete implies $\Delta$ is essentially self-adjoint on $C_c^\infty(M)$, so $L^2$-spectral theory can be used to define $\Delta^{-1/2}$, provided $0$ is not an eigenvalue of $\Delta$. Also, we use the metric to measure the size of the gradient in the definition of $T$.) There is now a vast literature on Riesz transforms on manifolds, which we do not review here. See \cite{ACDH}, \cite{HS2} and \cite{Nix} for further information and literature review.

The study of the Riesz transform of connected sums of manifolds originates in the seminal paper of Coulhon and Duong \cite{CD}. They showed that the Riesz transform on the connected sum of two copies of $\R^n$, $n >2$, is unbounded on $L^p$ for $p > n$. They also proved boundedness for $1 < p < 2$, leaving open the question of boundedness on $L^p$ in the range $2 < p \leq n$ for $n \geq 3$. This question was settled in the affirmative by Carron, Coulhon and the first author in \cite{CCH} who considered the connected sum of a finite number of copies of $\R^n$. 

In the last section of \cite{CCH}, several open problems were posed concerning the boundedness of the Riesz transform. In particular, the following question was posed: if several complete manifolds have Riesz transform bounded on $L^p$, under which conditions does the connected sum of such manifolds have Riesz transform bounded on $L^p$? This question was addressed by Carron in \cite{Ca2}, for manifolds satisfying a Sobolev inequality with volumes of balls $B(x, r)$ of radius $r$ growing like $r^\nu$ for $\nu$ strictly bigger than 3. 
Meanwhile, Grigoryan and Saloff-Coste \cite{GS99, GS} studied heat kernel estimates on `connected sums of Euclidean spaces of different dimensions', by which they meant connected sums of products $\R^{n_j} \times M_j$ of Euclidean spaces with compact manifolds of fixed total dimension $N$. 

The first and third authors in \cite{HS2} then generalized the result of \cite{CCH} to the class of  `connected sums of Euclidean spaces of different dimensions', provided that each $n_j$ is at least $3$. Let us define this class more  precisely. Let $\mathcal{V}_i$, for $i=1,\cdots, l$ be a family of complete connected non-compact Riemannian manifolds of the same dimension. We say that 
a Riemannian manifold 
$ \mathcal{V}$ is a connected sum of $\mathcal{V}_1,\ldots,\mathcal{V}_l$
and write 
\begin{equation}\label{defM}
\mathcal{V}=\mathcal{V}_1\# \mathcal{V}_2 \#\ldots \#\mathcal{V}_l
\end{equation}
if for some compact subset $K \subset \mathcal{V}$ the exterior 
$\mathcal{V}\setminus K$ is a disjoint union of connected open sets $\tilde{\mathcal{V}_i}$, $i=1,\cdots, l$, such that each $\tilde{\mathcal{V}_i}$ is isometric to $\mathcal{V}_i \setminus K_i$
for some compact sets $K_i \subset \mathcal{V}_i$. We call the subsets $\tilde{\mathcal{V}_i}$ the \emph{ends} of $\mathcal{V}$.

The main result of \cite{HS2} is 
 \begin{theorem}\label{thm:part1}
 Let $\R^{n_1} \times \mathcal{M}_1, \dots, \R^{n_l}\times \mathcal{M}_l$ be a set of $l \geq 2$ manifolds which are products of a Euclidean factor of dimension $n_i$ with a compact Riemannian manifold $\M_i$, with the product metric, for $1 \leq i \leq l$. 
 Suppose that  $\mathcal{M}=(\R^{n_1} \times \mathcal{M}_1) \# 
\ldots \# (\R^{n_l}\times \mathcal{M}_l)$ is a manifold with $l \geq 2$ ends in the sense of \eqref{defM}, with $n_i \geq 3$ for each $i$. 
Then  the Riesz transform   $T= \nabla \Delta^{-1/2}$ defined on $\M$ is bounded on $L^p(\M)$ if and only if 
$1 < p < \min\{n_1,\cdots, {n_l}\}$.
 In addition $T$ is of weak type 
 $(1,1)$.
\end{theorem}

The strategy employed in \cite{HS2} is similar to \cite{CCH}, although implemented slightly differently. In both cases, there is a `key lemma' which describes the behaviour of solutions (or approximate solutions) $u$ to the resolvent equation 
\begin{equation}
(\Delta + k^2) u = v, \quad v \in C_c^\infty(\M),
\label{res-uv}\end{equation}
in the low energy limit $k \to 0$, which are used to construct a parametrix at low energy for the resolvent kernel, in particular capturing some of its long-range behaviour as $k \to 0$. 

The solution to \eqref{res-uv} is obtained by perturbing off a suitable solution to the equation $\Delta u_0 = v$, at zero energy. One thing we can obtain from $u_0$, in the setting of \cite{CCH} and \cite{HS2}, is a harmonic function $\Phi$ that tends to a nonzero constant, say $1$, at one given end of our connected sum manifold $\M$, and to zero at all other ends. This is found by taking a cutoff function $\phi$ that is $1$ near infinity on one end, and $0$ near infinity on all other ends, letting $v = \Delta \phi$, and then defining $\Phi = v - u_{0}$. As $u_0$ tends to zero at infinity along all ends, this shows that $\Phi$ is nontrivial, with the asymptotic behaviour described above.  This function $\Phi$ shows up in the leading order long-range behaviour of the resolvent at low energy, far from the diagonal, and is crucial in understanding the range of $p$ for which the Riesz transform is $L^p$ bounded. 


In the present paper,  we analyze the connected sum of $\R^2 \times \M_-$ and $\R^{d_+} \times \M_+$ where $d_+ > 2$. This is in many ways the most interesting and challenging case. From the point of view of Brownian motion, we are connecting a recurrent space with a transient space. We obtain a space that is transient but `just barely', and this leads to several interesting non-uniformities in the resolvent kernel and heat kernel. In the case of the heat kernel, these were analysed quite precisely in \cite{GS} (although there are no gradient estimates in \cite{GS}, which would be needed to analyze the Riesz transform). Our goal in the present paper is to analyze the resolvent and Riesz transform \emph{without} using heat kernel estimates on $\M$. We do not completely eschew heat kernel estimates, but only use rather elementary estimates on the product manifolds $\R^{d_\pm}\times \M_\pm$; we never consider heat kernels on $\M$ itself. In fact, we would like to reverse the order of logic, and deduce long-time heat kernel asymptotics, as well as gradient bounds, from our resolvent estimates; this is something we hope to do in a future article. 

The main result of the present paper is 
 \begin{theorem}\label{thm:main}
 Let $\R^{n_-} \times \mathcal{M}_-$ and $\R^{n_+}\times \mathcal{M}_+$ be two manifolds which are products of a Euclidean factor of dimension $n_\pm$ with a compact Riemannian manifold $\M_\pm$, with the product metric. Suppose that  $\mathcal{M}=(\R^{n_-} \times \mathcal{M}_-)\# (\R^{n_+}\times \mathcal{M}_+)$ is a connected sum in the sense of \eqref{defM}, with $n_-=2$ and $ n_+ \geq  3$. 
	Then  the Riesz transform   $T$ defined on $\M$ is bounded on $L^p(\M)$ if and only if 
	$1 < p  \le 2 $.
	That is, there exists $C$ such that 
	$$\big\| \, |\nabla \Delta^{-1/2} f| \, \big\|_p\le C \|f\|_p,\ \forall f\in L^p(X,\mu)$$
	if and only if $1 < p \le 2 $.
	In addition   $T$ is of weak type 
	$(1,1)$.
\end{theorem}

\begin{remark} The analysis applies equally well to any finite number of ends, as in Theorem~\ref{thm:part1}. However, we have elected to restrict attention to just two ends, to make the proof as transparent as possible.
\end{remark}

We now give an indication of the differences in the analysis of `connected sums of Euclidean spaces of different dimensions' when one Euclidean factor has dimension two. 
The key step, as in \cite{CCH} and \cite{HS2}, is to prove the `key lemma' about solutions to \eqref{res-uv}. 
Our method of proof more closely resembles the analysis in \cite{CCH}. However, in \cite{CCH}, one could quote pre-existing results obtained from Melrose's b-calculus \cite{Melrose}  to show the existence, uniqueness and asymptotics of the function $u_0$ at $k=0$. In the present case, there is not, to our knowledge, any literature on harmonic functions on `connected sums of Euclidean spaces of different dimension'. Therefore, we devote Section~\ref{sec:connectedsums} to a treatment of such functions. It turns out that there are several key differences in the case where one of the Euclidean dimensions is equal to $2$. For example, in this setting, we show that there is no harmonic function $\Phi$ that tends to $1$ at the end with Euclidean dimension $2$ and tending to zero at all other ends. Nevertheless, we take a function $\phi$ as above (that is, equal to $1$ in a neighbourhood of infinity at the end with Euclidean dimension $2$, and supported away from infinity at the other end), and then solve equation \eqref{res-uv} with $v = - \Delta \phi$. We find that $u_{| k=0} = -\phi$, and has an expansion as $k \to 0$ in powers of $\ilg k := (\log 1/k)^{-1}$, which is a `non-classical' type of expansion that only appears when there is a Euclidean factor with dimension two\footnote{However, there is an expansion in $\ilg k$ in the case of the resolvent of a Schr\"odinger operator on a  four-dimensional asymptotically Euclidean manifold when the potential has a zero-resonance; see \cite{GH2}.}. 
Moreover, the coefficient of $\ilg k$ is a harmonic function $\mathcal{U}$ that has \emph{logarithmic growth} at the end with Euclidean dimension $2$ and tends to zero at the other end. This is then the analogue of the $\Phi$ function when the dimension is at least $3$. Moreover, the function $u$ appears in our low-energy parametrix, multiplied by a divergent factor of $\log k$, and thus $\mathcal{U}$ is present at order $O(1)$ in the parametrix, and as we show, in the true resolvent kernel as well. 

We have already discussed some of the direct antecedents of this research. 
We close this introduction by mentioning some other works closely related to this research question. 
Devyver \cite{Dev} and Jiang \cite{Ji1} found additional conditions, beyond Carron's work \cite{Ca1}, implying that the boundedness property of the Riesz transform on $L^p$ is stable under connected sums, or gluing. Grigor'yan and Saloff-Coste \cite{GS2} studied the closely related question of the Faber-Krahn inequality under connected sums. Guillarmou and the first author studied connected sums of asymptotically conic manifolds \cite{GH1}, \cite{GH2}. 
Bui, Duong, Li and Wick \cite{BDLW}  have studied the functional calculus for the Laplacian on `connected sums of Euclidean spaces of different dimensions'. The question of stability of the boundedness of the Riesz transform under compact perturbations has been considered by Coulhon and Dungey \cite{CDungey} and by Jiang and Lin \cite{Ji2}. A more comprehensive discussion of related literature is given in the second author's PhD thesis \cite{Nix}. 

This work constitutes part of the second author's doctoral studies. This research was supported by the Australian Research Council through Discovery Grant DP160100941 (A.H. and A.S.) and by the Australian Government through the award of a Australian Postgraduate Award to D.N.

\section{Connected sums}\label{sec:connectedsums}

%
%

\subsection{Resolvents on product spaces}

Let $\Delta_{\R^{n_\pm} \times \M_\pm}$ be the Laplacian on corresponding product space $\R^{n_\pm} \times \M_\pm$. In this section we derive estimates on the resolvent, and the gradient of the resolvent, from heat kernel estimates. This is in fact the only place where we shall use heat kernel estimates, and we only require elementary estimates on compact manifolds.

We denote  by $e^{-t\Delta_{\R^{n_\pm}\times\M_\pm}}(z,z')$ the heat kernel on the corresponding product space, 
where $z=(x,y)$, $z'=(x',y')$, $x,x' \in \R^{n_\pm}$ and  $y,y' \in \M_\pm$. Recall that the heat kernel on $\R^n$ has the explicit expression
\begin{equation*}
e^{-t\Delta_{\R^{n}}}(x,x') = \frac{1}{(4\pi t)^{n/2}} \exp \left(- \frac{|x-x'|^2}{4t} \right),
\end{equation*}
for $t>0, x,x' \in \R^n$. On compact manifolds, such as $\M_\pm$, of dimension $m$, we have Gaussian upper and lower bounds for small time \cite{Cheeger-Yau, CLY},
\begin{equation}
c t^{-m/2} \exp\Big( \frac{d(y, y')^2}{ct} \Big) \leq    e^{-t\Delta_{\M_\pm}}(y,y') \leq  C t^{-m/2} \exp\Big( \frac{d(y, y')^2}{Ct} \Big) , \quad 0 < t \leq 1, 
\end{equation}
while for large time the maximum principle gives us 
\begin{equation}
0 < c  \leq     e^{-t\Delta_{\M_\pm}}(y,y')  \leq  C < \infty , \quad  t \geq 1. 
\end{equation}

We also have the gradient estimates of the heat kernel
\begin{equation*}
|\nabla e^{-t\Delta_{\R^{n}}}(x,x')| = C_c\frac{1}{(4\pi t)^{(n+1)/2}} \exp \left(- \frac{|x-x'|^2}{4t} \right),
\end{equation*}
and 
\begin{equation*}
|\nabla e^{-t\Delta_{\M_\pm}}(y,y')| \leq 
\begin{cases}
Ct^{-(m+1)/2} \exp \left(- \frac{d(y,y')^2}{ct} \right), \qquad t \leq 1, \\
C e^{-\lambda_1 t} , \qquad t \geq 1,
\end{cases}
\end{equation*}
where $\lambda_1 > 0$ is the first positive eigenvalue of the Laplacian. 
Combining all of the above results gives heat kernel estimates and gradient estimates on the product space $\R^{n_\pm} \times \M_\pm$. Using the fact that the heat kernel on a product is a tensor product, 
\begin{equation*}
e^{-t\Delta_{\R^{n_\pm}\times \M_\pm}}((x,y), (x',y')) = e^{-t\Delta_{\R^{n_\pm}}}(x,x') e^{-t\Delta_{\M_\pm}}(y,y'),
\end{equation*}
we obtain (with $z = (x,y)$ a coordinate on the product) the upper bound, 
\begin{equation*}
e^{-t\Delta_{\R^{n_\pm}\times M_\pm}}(z,z') \leq C(t^{-n_\pm/2}+ t^{-N/2})\exp \left(- \frac{|z-z'|^2}{Ct} \right)
\end{equation*}
the lower bound 
\begin{equation*}
 c(t^{-n_\pm/2}+ t^{-N/2})\exp \left(- \frac{|z-z'|^2}{ct} \right) \leq e^{-t\Delta_{\R^{n_\pm}\times M_\pm}}(z,z')
\end{equation*}
and the gradient estimate 
\begin{equation*}
|\nabla e^{-t\Delta_{\R^{n_\pm}\times M_\pm}}(z,z')| \leq C(t^{-(n_\pm+1)/2}+ t^{-(N+1)/2})\exp \left(- \frac{|z-z'|^2}{ct} \right),
\end{equation*}
where $N=\dim \mathcal{M}_- + n_- = \dim \mathcal{M}_+ + n_+$.

We denote the resolvent kernel on the $\pm$ end by $(\Delta_{\R^{n_\pm} \times \M_\pm} + k^2)^{-1}(z,z')$. We use the identity 
\begin{equation} \label{heattoresolvent}
(\Delta_{\R^{n_\pm} \times \M_\pm} + k^2)^{-1} = \int_0^\infty e^{-tk^2} e^{-t\Delta_{\R^{n_\pm} \times \M_\pm}} \ dt
\end{equation}
to obtain bounds on the resolvent kernel from heat kernel bounds.

Consider the ordinary differential equation
\begin{equation*}
f'' + \frac{a-1}{r}f' - f = 0,
\end{equation*}
for $a \geq 1$. There is a positive solution decaying as $r \to \infty$  given by
\begin{equation*}
L_a(r) = r^{1-a/2}K_{|a/2-1|}(r), 
\end{equation*}
for $a \geq 1$. Here $K_{|a/2-1|}$ is the modified Bessel function of the second kind. The asymptotics of $L_a$, for $a>2$, are given by
\begin{equation*}
L_a(r) \sim 
\begin{cases}
r^{2-a}, \qquad &r \leq 1 \\
r^{(1-a)/2}e^{-r}, \qquad &r \geq 1.
\end{cases}
\end{equation*}
The asymptotics of $L_2(r) = K_0(r)$, i.e. when $a=2$, are given by
\begin{equation*}
L_2(r) \sim 
\begin{cases}
-\log r + \log 2 - \gamma \qquad & r < 1 \\
\frac{\sqrt{\pi}}{\sqrt{2}} \frac{e^{-r}}{\sqrt{r}} \qquad & r \geq 1,
\end{cases}
\end{equation*}
where $\gamma$ is the Euler-Mascheroni constant. Furthermore, for $a \geq 3$,
\begin{equation}
C_{a,c}r^{2-a}e^{-r} \leq L_a(r) \leq C_a' r^{2-a} e^{-cr}, \qquad 0 < c < 1,
\end{equation}
while for $a=2$ we have 
\begin{equation}
C_{\tilde c}(1 + |\log r|)e^{-\tilde c r} \leq L_a(r) \leq C (1 + |\log r|) e^{-r}, \qquad \tilde c > 1.  
\end{equation}

Furthermore, the following identity holds for $a \geq 1$,
\begin{equation} \label{a>=1identity}
\int_0^\infty e^{-tk^2} \, t^{-a/2} \, \exp \left(- \frac{r^2}{4t} \right) \ dt = C_a k^{a-2} L_a(kr).
\end{equation}
for some constant $C_a > 0$.
Applying $(\ref{heattoresolvent})$ and $(\ref{a>=1identity})$ yields the resolvent kernel estimates and gradient estimates on each product space. For the $+$ case  we have an upper bound 
\begin{equation}
(\Delta_{\R^{n_+} \times \M_+} + k^2)^{-1}(z,z') \leq C \left(d(z,z')^{2-N} + d(z,z')^{2-n_+} \right)\exp \left(-ckd(z,z') \right),
\label{resolvent-3}\end{equation}
a lower bound 
\begin{equation}
(\Delta_{\R^{n_+} \times \M_+} + k^2)^{-1}(z,z') \geq c \left(d(z,z')^{2-N} + d(z,z')^{2-n_+} \right)\exp \left(-Ckd(z,z') \right),
\end{equation}
and a gradient estimate
\begin{equation}
\left| \nabla(\Delta_{\R^{n_+} \times \M_+} + k^2)^{-1}(z,z') \right| \leq C \left(d(z,z')^{1-N} + d(z,z')^{1-n_+} \right)\exp \left(-ckd(z,z') \right).
\label{gradresolvent-3}\end{equation}
For the two-dimensional case $\R^2 \times \M_-$, and for $k > 0$,  we have an upper bound 
\begin{equation}
(\Delta_{\R^{2} \times \M_-} + k^2)^{-1}(z,z') \leq C \left[ d(z,z')^{2-N} + 1 + | \log\left(k d(z,z')\right) |  \right]\exp \left(-ckd(z,z') \right),
\label{resolvent-2}\end{equation}
a lower bound
\begin{equation}\label{res-lowerbound-2}
(\Delta_{\R^{2} \times \M_-} + k^2)^{-1}(z,z') \geq c \left[ d(z,z')^{2-N} + 1 + | \log\left(k d(z,z')\right) |  \right]\exp \left(-Ckd(z,z') \right),
\end{equation}
and a gradient bound
\begin{equation} \label{gradresolvent-2}
\left| \nabla(\Delta_{\R^{2} \times \M_-} + k^2)^{-1}(z,z') \right| \leq C \left(d(z,z')^{1-N} + d(z,z')^{-1}  \right)\exp \left(-ckd(z,z') \right).
\end{equation}

\begin{remark}
One phenomenon that occurs in dimension $2$ is that the Euclidean Green function $-2\pi \log |x-x'|$ is not the pointwise limit, as $k \to 0$, of the kernel of $(\Delta_{\R^2} + k^2)^{-1} (x,x')$. In fact, the latter kernel has a $-\log k$ divergence as $k \to 0$, which must be subtracted in order to obtain a pointwise limit. This explains the somewhat paradoxical fact that the Green function on $\R^2$ changes sign, despite the fact that the kernel of $(\Delta + k^2)^{-1}$ on any complete manifold is positive for all $k > 0$. 
\end{remark}


\subsection{Harmonic functions on the ends of $\M$}\label{subsec:harmonic}
We next consider functions on $\M$ that are harmonic outside a compact set. This is preparation for defining an elliptic boundary problem on the compact part $K$ of $\M$, which will allow us to solve the inhomogeneous Laplace equation globally on $\M$. We shall treat the two ends $E_\pm$ separately, as the end $E_-$ with two-dimensional Euclidean factor behaves differently to $E_+$. 

Let $K \subset \M$ be a compact set such that the boundary of $K$ is contained in the two ends (that is, where the metric is of product type), and such that the boundary $\partial K$ is a disjoint union of the form $\partial_- K = \{ x \in \R^2 \mid |x| = R \} \times \M_-$ in the $E_-$-end, and $\partial_+ K = \{ x \in \R^{n_+} \mid |x| = R \}  \times \M_+$ in the $E_+$-end. Given a smooth function $f$ on $\partial_\pm K$, we wish to extend it to a harmonic function on $E_\pm$ with `good' behaviour at infinity. 

Below, we denote by $r$ a function on $\M$ that is equal to the Euclidean radial coordinate $|x|$ on each end $E_\pm$, and is $\geq 1$ everywhere on $\M$. The radial derivative $\partial_r$ is taken to be the usual radial derivative in Euclidean polar coordinates on $E_\pm$, and is not defined on the interior of $K$. 

\begin{lemma}\label{lem:harmonic-} On $\partial_- K$, there is a unique harmonic extension operator 
$$
\E_- : C^\infty(\partial_- K) \ni f \mapsto u = \E_- f \in C^\infty(E_-)
$$
such that $u$ tends to a constant at infinity and has a complete asymptotic expansion in nonpositive powers of $r$ as $r \to \infty$. Moreover, the radial derivative $\partial_r u$ is $O(r^{-2})$ at infinity and has a complete asymptotic expansion obtained by differentiating the expansion for $u$ term-by-term. 
\end{lemma}

\begin{proof}
We use polar coordinates $(r, \theta)$ on $\R^2 \setminus \{ |x| \leq R \}$, and expand in the eigenfunctions $\psi_l^-$ on $\M_-$, $l \geq 0$. Let $\mu_l^2$ be the eigenvalue of the Laplacian $\Delta_{\M_-}$ corresponding to $\psi_l^-$. Notice that $\mu_0 = 0$ and $\mu_l > 0$ for all $l > 0$. 
We expand $f$ simultaneously in a Fourier series in $\theta$ and in the $\psi_l^-$: 
\begin{equation*}
f = \sum_{m \in \Z,l\in \N} c_{ml} e^{im\theta} \psi_l^-(y)
\end{equation*}
where, thanks to the smoothness of $f$, we have $c_{ml} = O(\langle m \rangle^{-M} \langle l \rangle^{-M})$, for all $M \in \N$. 

Separating variables we seek a bounded solution $u$ to $\Delta u = 0$ in $E_-$ with boundary value $f$, of the form  
\begin{equation}
u(r, \theta, y) = \sum_{m \in \Z, l\in \N} b_{ml}(r) e^{im\theta} \psi_l^-(y).
\label{u-exp-2}\end{equation}
To solve this equation we require that 
\begin{equation}
\left(-\frac{\d^2}{\d r^2} - \frac{1}{r} \frac{\d}{\d r} + \frac{m^2}{r^2} + \mu_l^2 \right) b_{ml}(r) = 0.
\label{b-eqn}\end{equation}
It is convenient to split into the cases $l = 0$ and $l > 0$. 

$\underline{l\neq0:}$ Let $b_{ml}(r) = \tilde{b}_{ml}(\mu_l r)$ for $l \neq 0$. The above equation becomes
\begin{equation*}
\left(-\frac{\d^2}{\d r^2} - \frac{1}{r} \frac{\d}{\d r} + \frac{m^2}{r^2} + 1 \right) \tilde{b}_{ml}(r) = 0.
\end{equation*}
Observe that this is the modified Bessel equation of order $m$ which has solutions that are linear combinations of $I_{|m|}(r)$ and $K_{|m|}(r)$. In order to satisfy the boundedness condition on $u$ we must choose the exponentially decaying solution, so we take $\tilde{b}_{ml}(r) = cK_{|m|}(r)$, that is, 
\begin{equation}
b_{ml}(r) = c K_{|m|}(\mu_l r).
\label{bmlneq0}\end{equation}
We now treat the $l=0$ case. \\
$\underline{l=0:}$
\begin{equation*}
\left(-\frac{\d^2}{\d r^2} - \frac{1}{r} \frac{\d}{\d r} + \frac{m^2}{r^2}  \right) b_{m0}(r) = 0.
\end{equation*}
Solving we obtain solutions $r^{\pm m} \text{ for } m \neq 0$ and $1 \text{ or }  \log r  \text{ for } m=0$. In order to satisfy the boundedness condition, we take
\begin{equation}
\tilde{b}_{m0}(r) = \begin{cases}
c r^{- |m|} \qquad \text{for } m \neq 0 \\
c  \qquad \qquad \text{for } m=0.
\end{cases}
\label{bml=0}\end{equation}
Thus, a harmonic extension of $f$ to the end $E_-$ is given formally by the sum 
\begin{equation}\label{u-exp-exp}
u(r, \theta, y) = \sum_{m \in \Z} c_{m0} e^{im\theta} \psi^-_0(y) \big( \frac{r}{R} \big)^{-|m|} + 
\sum_{m \in \Z,l \geq 1} c_{ml}  e^{im\theta} \psi_l^-(y) \frac{K_{|m|}(\mu_l r)}{K_{|m|}(\mu_l R)}
\end{equation}
which matches with $f$ at $r=R$. Notice that in the first term on the RHS, $\psi^-_0(y)$ is a constant, equal to $(\vol(\mathcal{M}_-))^{-1/2}$. 

We next consider the convergence of this sum. Note that $\sup_y |\psi^-_l(y)|$ has at most polynomial growth in $l$, thanks to the estimate $\| \psi^-_l \|_{L^\infty} \leq C \mu_l^{(\dim \mathcal{M}_- - 1)/2}$ and Weyl asymptotics for the $\mu_l$. Taking into account the rapid decrease of $c_{ml}$ and the fact that 
the functions 
$$
\big( \frac{r}{R} \big)^{-|m|} \quad \text{ and } \quad \frac{K_{|m|}(\mu_l r)}{K_{|m|}(\mu_l R)}
$$
are uniformly bounded by $1$, for $r \geq R$,  this shows that the series converges uniformly. As a uniform limit of harmonic functions, the limit is also harmonic. This shows that the expression \eqref{u-exp-exp} defines a harmonic function with boundary value $f$. 

We next show that 
$u$ has an asymptotic expansion in nonpositive powers of $r$ given by the $l=0$ terms. More precisely, we have 
\begin{equation}\label{u-asymptexp}
u = \vol(\mathcal{M}_-)^{-1/2} \sum_{|m| \leq M} c_{m0} e^{im\theta}  \big( \frac{r}{R} \big)^{-|m|} + O(r^{-M-1}), \quad r \to \infty, 
\end{equation}
for every integer $M \geq 0$. 

To demonstrate  asymptotic expansion \eqref{u-asymptexp} we need to establish two estimates. The first is to show that the sum of the tail of the $l=0$ series, that is, 
$$
\vol(\mathcal{M}_-)^{-1/2} \sum_{|m| > M} c_{m0} e^{im\theta}  \big( \frac{r}{R} \big)^{-|m|},
$$
is $O(r^{-M-1})$ as $r \to \infty$. The second is to show that the sum of the $l \geq 1$ terms decays faster than any polynomial. In fact, we shall show that the sum of the $l \geq 1$ terms is exponentially decaying. 

The first estimate is straightforward. We remove a factor of $(r/R)^{-M-1}$ from each term, obtaining 
$$
\big( \frac{r}{R} \big)^{-M-1} S(r, \theta, y), \quad S(r, \theta, y) =  \vol(\mathcal{M}_-)^{-1/2} \sum_{|m| \geq M+1} c_{m0} e^{im\theta} \big( \frac{r}{R} \big)^{-|m| +M+1},
$$
and the argument is the same as before: a uniform bound on $S$ is given by 
$$
\vol(\mathcal{M}_-)^{-1/2} \sum_{|m| \geq M+1} |c_{m0}| ,
$$
which is finite due to the rapid decrease  of $c_{m0}$. The same reasoning shows that if we take the partial $r$-derivative of each term, the resulting series also converges uniformly.

For the second estimate, we use the following standard integral representation of modified Bessel functions $K_\nu$ for $\nu \geq 0$:
\begin{equation}\label{Knu-int}
K_\nu(y) = \int_0^\infty e^{-y \cosh t} \cosh (\nu t) \, dt, \quad \nu > 0, \quad y > 0. 
\end{equation}
Now suppose that $0 < x < y$. Then, using the elementary inequality $e^{-y \cosh t} \leq e^{-x \cosh t} e^{x-y}$, and substituting into \eqref{Knu-int}, we find that 
\begin{equation}\label{exp-ineq}
K_\nu(y) \leq e^{x-y} K_\nu(x), \quad \text{ provided } 0 < x < y. 
\end{equation}
We use this to estimate the size of the $l \geq 1$ sum in \eqref{u-exp-exp}. Let $c$ be a number smaller than $\mu_1$. Using \eqref{exp-ineq}, we have 
$$
\frac{K_{|m|}(\mu_l r)}{K_{|m|}(\mu_l R)} \leq e^{-\mu_l (r-R)}, \quad r \geq R.
$$
Using this, 
we can pull out a factor $e^{-c(r-R)}$ in front of the sum, and obtain 
\begin{equation}\label{unif-exp-decay}
\Big| \sum_{m \in \Z,l \geq 1} c_{ml}  e^{im\theta} \psi_l^-(y) \frac{K_{|m|}(\mu_l r)}{K_{|m|}(\mu_l R)} \Big| \leq e^{-c(r-R)} \sum_{m \in \Z,l \geq 1} |c_{ml}|e^{im\theta} \| \psi_l^- \|_{L^\infty} ,
\end{equation}
and again the rapid decrease of the $c_{ml}$ and the at most polynomial increase of $\| \psi_l^- \|_{L^\infty}$ show that the RHS is bounded by a constant times $e^{-c(r-R)}$. This completes the proof of \eqref{u-asymptexp}.

Next, we show that $\partial_r u$ has the corresponding expansion obtained by term-by-term differentiation.  We use the well known fact that if a series $\sum_j v_j$ converges uniformly to $u$, and the series of radial derivatives $\sum_j \partial_r(v_j)$ converges uniformly to $w$, then $\partial_r u$ exists and is equal to $w$. Because of this, it only remains to show that after taking the radial derivative term-by-term, the tail of this new series satisfies similar estimates as above. This is trivial for the $l=0$ terms. For the $l \geq 1$ terms, we need a uniform estimate on the derivatives of modified Bessel functions. To do this, we use another standard identity: 
\begin{equation*}
K_m(r) = \frac{\sqrt{\pi}r^m}{2^m \Gamma(m+\frac12)} \int_1^\infty e^{-rt} (t^2-1)^{m - \frac12} \ dt.
\end{equation*}
By changing variable to $r(t-1)$, we can write this in the form
\begin{equation*}
K_m(r) = \frac{\sqrt{\pi}r^{-1/2} e^{-r}}{2^m \Gamma(m+\frac12)} \int_0^\infty e^{-t} (\frac{t^2}{r} + 2t)^{m - \frac12} \ dt.
\end{equation*}
Differentiating in $r$, we find that 
\begin{equation}\begin{gathered}
rK_m'(r) = - \frac1{2} K_m(r) - r K_m(r) + \frac{\sqrt{\pi}r^{-1/2} e^{-r}}{2^m \Gamma(m+\frac12)} (m-\frac1{2})\int_0^\infty e^{-t} (-\frac{t^2}{r})(\frac{t^2}{r} + 2t)^{m - \frac32} \ dt \\
\geq - \frac1{2} K_m(r) - r K_m(r) - \frac{\sqrt{\pi}r^{-1/2} e^{-r}}{2^m \Gamma(m+\frac12)} (m-\frac1{2})\int_0^\infty e^{-t} (\frac{t^2}{r} + 2t)^{m - \frac12} \ dt \\
= - m K_m(r) - r K_m(r). 
\end{gathered}\end{equation}
Hence, noting that $K_m$ is a decreasing function (which is evident from the first line above), thus $K_m'$ is negative, we have 
\begin{equation}\label{K'-est}
\Big| r K_m'(r) \Big| \leq m K_m(r) + r K_m(r).
\end{equation}
Now applying this to the series of radial derivatives for $l \geq 1$, we are in effect reduced to the previous estimate apart from an addition factor of $|m|$ or $\mu_l$, which is of no significance due to the rapid decrease of the $c_{ml}$. 

We now show uniqueness of $u$. Let $\Pi_m$ be the spectral projector onto the $m^2$-eigenspace of $\Delta_S$ and let $\pi_l$ be the spectral projector onto the $\mu_l^2$-eigenspace of $\Delta_{\M_-}$. If $u$ is harmonic, for $r \geq R$, then since $\Delta$ commutes with rotations on $\R^2$ and $\Delta_{\M_-}$ on $\M_-$, it commutes with spectral projections of these. So, if $u$ is harmonic, then $u_{ml}:= \Pi_m \pi_l u = b_{ml}(r)e^{im\theta} \psi_l^-(y)$ is harmonic. Boundedness of $u$ implies that $|b_{ml}(r)| \leq C$. Then $b_{ml}(r)$ solves the ordinary differential equation \eqref{b-eqn}. In order to satisfy the boundedness condition, we must take the solutions \eqref{bmlneq0} or \eqref{bml=0} according to $\l \neq 0$ or $l=0$ respectively, with the constant determined by the matching condition with $f$, as discussed previously. This shows uniqueness of $u$.

This concludes the proof of the Lemma. 
\end{proof}


A similar operator is defined on $\partial_+ K$, with the difference that here, we ask that the solution decays to zero at infinity. 

\begin{lemma}\label{lem:harmonic+} On $\partial_+ K$, there is a unique  extension operator 
$$
\E_+ : C^\infty(\partial_+ K) \ni f \mapsto u = \E_+ f \in C^\infty(E_+)
$$
such that $u$ is harmonic on $E_+$ and  is $O(r^{2-n_+})$ at infinity. Moreover $u$  has a complete asymptotic expansion in negative powers of $r$ as $r \to \infty$. In addition,  the radial derivative $\partial_r u$ is $O(r^{1-n_+})$ at infinity and has a complete asymptotic expansion obtained by differentiating the expansion for $u$ term-by-term. 
\end{lemma}

\begin{proof} The proof proceeds in exactly the same way as the previous proof, so we only indicate the points of difference. 

We use polar coordinates $(r, \theta)$ on $\R^{n_+} \setminus \{ |x| \leq R \}$, where $\theta \in S^{n_+ - 1}$, and expand in the eigenfunctions $\psi_l^+$ on $\M_+$, $l \geq 0$. Let $\nu_l^2$ be the eigenvalue of the Laplacian $\Delta_{\M_+}$ corresponding to $\psi_l^+$. Notice that $\nu_0 = 0$ and $\nu_l > 0$ for all $l > 0$. 
We expand $f$ simultaneously in the $\psi_l^+$ on $\M_+$ and in spherical harmonics $\phi_{j,m}(\theta)$ in $\theta$, where $\Delta_{S^{n_+-1}} \phi_{j,m} = m (n_+-2+m) \phi_{j,m}   $; the dimensional of this eigenspace grows polynomially with $m$. Thus we have  
\begin{equation*}
f = \sum_{j,m,l} c_{jml} \phi_{j,m}(\theta)  \psi_l^+(y)
\end{equation*}
where, thanks to the smoothness of $f$, we have $c_{jml} = O(\langle m \rangle^{-N} \langle l \rangle^{-N})$, for all $N$. 

Separating variables, as above, we seek a bounded solution $u$ to $\Delta u = 0$ in $E_+$ with boundary value $f$, in the form  
\begin{equation*}
u = \sum_{m,l} b_{jml}(r) \phi_{j,m}(\theta)  \psi_l^+(y).
\end{equation*}
The case $l > 0$ is very similar to the case above. We find that $b_{jml}(r)$ is $r^{-(n_+ - 2)/2}$ times a modified Bessel function of order $(n_+-2)/2 + m$. 
The case $l=0$ is slightly different. 
To solve this equation we require that 
\begin{equation*}
\left(-\frac{\d^2}{\d r^2} - \frac{n_+ - 1}{r} \frac{\d}{\d r} + \frac{m(n_+-2+m)}{r^2} b_{jm0}(r) \right) = 0.
\end{equation*}
This has solutions $r^{m}$ and $r^{-(n_+-2) - m}$. To comply with the condition that the solution should vanish in the limit $r \to \infty$ we must choose a multiple of $r^{-(n_+-2) - m}$. In this way, we obtain a unique solution $u$ as above.  

The proof of the asymptotic expansion is made in exactly the same way as for $E_-$. 
\end{proof}

We next use these extension operators to define exterior Dirichlet-to-Neumann operators. For this purpose, let $\nu = -\partial_r$ be the unit vector field along $\partial K$ pointing in the direction of the interior of $K$. 

\begin{definition}\label{def:DtN} We define, for each end $E_\pm$, an exterior Dirichlet-to-Neumann operator $\Lambda_{\textup{ext},\pm} : C^\infty(\partial_\pm K) \to 
C^\infty(\partial_\pm K)$ as follows:
$$
\Lambda_{\textup{ext},\pm} f = \partial_\nu \E_\pm f \big|_K.
$$
That is, $\Lambda_{\textup{ext},\pm} f$ is the normal derivative of the harmonic extension $u$ of $f$ constructed in Lemmas~\ref{lem:harmonic-} and \ref{lem:harmonic+}. 
\end{definition}

The following result is well-known, but as we did not locate a proof valid in the exterior setting, we provide a quick proof. 

\begin{proposition}
The exterior Dirichlet-to-Neumann operators $\Lambda_{\textup{ext},\pm}$ are pseudodifferential operators or order $1$, with symbol $\sigma(\Lambda_{\textup{ext},-})(z, \zeta) = |\zeta|$, where $|\zeta|$ denotes the length of $\zeta$ in the product metric on $E_-$, restricted to $\partial_- K$. 
\end{proposition}

\begin{proof}
We just prove this for the $E_-$ end as the other case is analogous. 

In terms of the series expansion \eqref{u-exp-exp} for $u = \mathcal{E}_{\pm} f$, we have 
\begin{equation}
\Lambda_{\textup{ext},-} f =  
\vol(\mathcal{M}_-)^{-1/2} \sum_m |m| c_{m0} e^{im\theta} + 
\sum_{m \in \Z,l \geq 1} -c_{ml}\frac{ \mu_l K_{|m|}'(\mu_l)}{K_{|m|}(\mu_l)}  e^{im\theta} \psi_n^-(y) .
\label{Lambda-sum}\end{equation}
The convergence of this sum  follows in a similar way to the proofs above, using \eqref{K'-est} and the rapid decrease of $c_{ml}$. Moreover, differentiating in either $\theta$ or $y$ brings down powers of $m$ (in the case of $\theta$) or increases the sup norm in the $y$ variable by powers of $\mu_l$. Due to the rapid decrease of the $c_{ml}$, the series remains uniformly convergent under repeated differentiation in $(\theta, y)$, and therefore converges to a $C^\infty$ function. 

To show that $\Lambda_{\textup{ext},-}$ is a pseudodifferential operator of order $1$, we use the fact that this is true on a compact manifold with boundary. We can compare $\Lambda_{\textup{ext},-}$ with the Dirichlet-to-Neumann operator on a compact manifold with boundary by compactifying $E_- $ stereographically at infinity, leaving the metric unchanged in a neighbourhood of $\partial_- K$. Let us denote the resulting compact Riemannian manifold with boundary by $\tilde E_-$. It has an associated  DtN operator $\tilde{\Lambda}$ which is a pseudodifferential operator of order $1$ --- see \cite[Section 11, Chapter 7]{Taylor2}. Moreover, the principal symbol of $\tilde{\Lambda}$ is $|\zeta|$, as follows from \cite[Chapter 7, equations (11.8) and (11.35)]{Taylor2}. We claim that the Schwartz kernel of $\tilde{\Lambda} - \Lambda_{\textup{ext},-}$ is smooth. To see this, take any distribution $f$ on $\partial_- K$, and let $\tilde u$ and $u$ be the bounded extensions to $\tilde E_-$ and $E_-$, respectively. Then $\tilde u - u$ is harmonic near $\partial_-K$ with zero boundary condition, hence $C^\infty$. It follows that $\Lambda_{\textup{ext},-}$ differs from $\tilde \Lambda$ by a smooth Schwartz kernel. Therefore, $\Lambda_{\textup{ext},-} \in \Psi^1(\partial_-K)$ is also a pseudodifferential operator of order $1$, with the same principal symbol as $\tilde{\Lambda}$. 
\end{proof}


\subsection{An elliptic boundary value problem}

In preparation for the key lemma below, we set up an elliptic boundary problem on the compact part $K$ of our manifold $\M$. This will allow us to solve Laplace's equation globally on $\M$ as well as find global harmonic functions satisfying certain asymptotics at infinity. 


Let $K$ be as in Section~\ref{subsec:harmonic}. 
Given $F \in L^2(\M)$ with support contained in the interior of $K$, we seek a solution to the equation 
\begin{equation}
\Delta u = F
\label{Laplace}\end{equation}
satisfying the boundary condition
\begin{equation}
Bu := \partial_\nu u - \Lambda_{\textup{ext},\pm} (u |_{\partial_\pm K}) = 0 \text{ at } \partial_\pm K
\label{Dtn-bc}\end{equation}
where we recall that $\nu = -\partial_r$ is the inward-pointing unit normal vector field at $\partial K$. 

The point of this boundary problem is that, if $F$ is as above and $u$ is a solution, then $u$ extends to a function on the whole manifold $\M$ that is harmonic on $E_\pm$, bounded on the end $E_-$, and decaying on the end $E_+$. In fact, we have already checked that, given the boundary values $u |_{\partial_\pm K}$ of $u$,  there is a unique harmonic function on the ends with these boundary values and which is bounded on the $E_-$ end and decaying to zero on the $E_+$ end. Then the normal derivative of the extension $u_\pm$ on $E_\pm$ is given by $\Lambda_{\textup{ext},\pm} u |_{\partial_\pm K}$, which, by construction, agrees with the normal derivative of the solution inside $K$. Since both $u_\pm$ and $u$ are harmonic in a neighbourhood of $\partial_\pm K$ (due to the support assumption on $F$), and their Cauchy data agree, they glue together to form a  function that is smooth across $\partial K$ and harmonic outside the support of $F$. 

To solve such a boundary problem we need to show

\begin{lemma}\label{lem:elliptic}
The equation \eqref{Laplace} with boundary condition \eqref{Dtn-bc} is elliptic in the sense of H\"ormander. 
\end{lemma}

\begin{proof}
Since $\Delta$ is clearly elliptic, it only remains to check that the boundary condition is elliptic.  
Set $D_t=-i\frac{\partial}{\partial_t}$. 
Recall that a second order PDE $Pu = f$ with boundary condition $B u = 0$, where $B = \Op(b)$ is a differential operator in the normal variable with coefficients that are pseudodifferential operators in the tangential variables, is elliptic if $P$ is elliptic and if 
 for all $x' \in \d K, \xi \in T_{x'}^*(\d K)$, where $\xi$ is not proportional to the interior conormal $n_{x'}$, we have that
\begin{equation*}
b(x', \xi+D_t n_{x'})u(0): M_{x',\xi}^+ \rightarrow \C
\end{equation*}
is bijective. Here $M_{x',\xi}^+ = \{ u \in C^\infty(\R_+): u \text{ bounded}, \; p(x',\xi+D_t n_x)u(t) = 0 \}$, $p$ is the principal symbol of $P$ and $b$ is the principal symbol of $B$.

In our situation, $b(x', (\xi', \xi_n)) = i \xi_n - |\xi'|$ and $p(x', \xi) = |\xi'|^2 + \xi_n^2$, where $\xi = (\xi_1,\ldots,\xi_n) = (\xi',\xi_n)$. The set $M_{x',\xi}^+$ consists of the bounded solutions to
\begin{equation*}
-u''(t)-2i\xi_n u'(t) + (|\xi'|^2 + \xi_n^2)u(t) = 0.
\end{equation*}
Letting $v = e^{i \xi_n t}u$ implies
\begin{equation*}
-v''(t) = e^{i \xi_n t} \left(-u''(t)-2i\xi_n u'(t) + \xi_n^2 u(t) \right) = -|\xi'|^2 v(t).
\end{equation*}
So $v(t) = Ce^{-|\xi'|t}$, where $C$ is a constant.
Hence, $M_{x',\xi}^+$ is the set of multiples of the function $e^{-(|\xi'| + i\xi_n) t}$, and we need to check that applying $b(x', \xi+D_t n_{x'})$ does not kill this function. 
Clearly 
$$
b(x', \xi+D_t n_{x'})  e^{-(|\xi'| + i\xi_n) t} = -2 |\xi'| e^{-(|\xi'| + i\xi_n) t}
$$
which never vanishes when $\xi' \neq 0$. This verifies that \eqref{Laplace} with boundary condition \eqref{Dtn-bc} is indeed an elliptic boundary problem. 
\end{proof}

\begin{remark} We have used the definition of elliptic boundary problem from \cite[Chapter 20, Definition 20.1.1]{Ho3}. In this definition, it is assumed that the boundary operators $B_j$ are differential. However, the theory makes perfect sense if the $B_j$ are pseudodifferential operators on the boundary, as remarked on \cite[p242]{Ho3}. See also \cite[Chapter 7, Sections 11 and 12]{Taylor2}, where pseudodifferential boundary conditions are considered.
\end{remark}

As a consequence of the theory of elliptic boundary value problems, the map 
\begin{equation*}
(Pu, Bu): {H}^2(K) \rightarrow L^2(K) \oplus H^{\frac12}(\partial K)
\end{equation*}
is a Fredholm operator --- see \cite[Chapter 20, Theorem 20.1.2]{Ho3}. It follows that, restricting to functions $u$ satisfying $Bu = 0$, the map 
\begin{equation}
\{ u \in H^2(K) \mid Bu = 0 \} \mapsto \Delta u \in L^2(K)
\label{Delta-Dtn}\end{equation}
is also Fredholm. 

We now claim further that the map $A$ in \eqref{Delta-Dtn} (that is, with domain $\{ u \in H^2(K) \mid Bu = 0 \}$) is invertible. This proceeds in two steps. The first step is to show that $A$ is self-adjoint. To see this, first recall that $\Lambda_{\textup{ext},\pm} $ is self-adjoint on $L^2(\partial K)$ with domain $H^1(\partial K)$. This implies that $A$ is a symmetric operator: for $u, v \in \dom A$, we have 
\begin{multline*}
\langle \Delta u, v \rangle - \langle u, \Delta v \rangle = \langle \partial_\nu u, v \rangle_{\partial K} -  \langle u, \partial_\nu v \rangle_{\partial K} 
= \langle \Lambda_{\textup{ext},\pm} u, v \rangle_{\partial K} -  \langle u, \Lambda_{\textup{ext},\pm} v \rangle_{\partial K} = 0. 
\end{multline*}
Thus the domain of $A^*$ contains the domain of $A$. 

Conversely, suppose $v \in \dom A^*$. Then,  there exists $\psi \in L^2(K)$ satisfying
$$
\langle \Delta u, v \rangle = \langle u, \psi \rangle
$$
for all $u \in \dom A$ (and then we define $A^*v = \psi$). By elliptic regularity, this holds only if $\Delta v \in L^2(K)$, which implies that $v \in H^2(K)$. For such $v$ we have, using Green's formula,
$$
\langle \Delta u, v \rangle = \langle u, \Delta v \rangle + \langle \partial_\nu u, v \rangle_{\partial K} -  \langle u, \partial_\nu v \rangle_{\partial K}
$$
where the last two inner products are in $L^2(\partial K)$. Using the boundary condition and the self-adjointness of $\Lambda_{\textup{ext},\pm} $ we write this as 
$$
\langle u, \psi \rangle = \langle \Delta u, v \rangle = \langle u, \Delta v \rangle + \langle  u, \Lambda_{\textup{ext},\pm} v - \partial_\nu v \rangle_{\partial K} .
$$
This can only hold for all $u \in \dom A$ if $\Delta v = \psi$ and $\Lambda_{\textup{ext},\pm} v - \partial_\nu v = 0$ at $\partial_\pm K$; that is, if $v \in \dom A$. 
This shows that $A$ is self-adjoint. 

The second step is to prove invertibility. Since $A$ is self-adjoint, its index is zero. So it is invertible if and only if we can show

\begin{lemma}\label{lem:uniqueness}
The homogeneous problem $\Delta u = 0$ for $u \in \dom A$ has only the trivial solution. 
\end{lemma}

\begin{proof}
Suppose $u \in \dom A$ solves $\Delta u = 0$. Then, as discussed above Lemma~\ref{lem:elliptic},  $u$ extends globally to $\M$ to a bounded harmonic function, decaying to zero on the $E_+$ end. 
By elliptic regularity, $u$ is $C^\infty$. 

Let $K_r$ be the set 
$$
K_r = K \cup \{ (x, y) \in E_- \mid |x| \leq r \} \cup \{ (x, y) \in E_+ \mid |x| \leq r \}, \quad r > R. 
$$
We compute, for $r > R$, 
\begin{equation}
\int_{K_r} |\nabla u|^2 \, dg = \int_{K_r} \Delta u \overline{u} \, dg - \int_{\partial K_r} \partial_\nu u \overline{u} \, d\sigma = - \int_{\partial K_r} \partial_\nu u \overline{u} \, d\sigma
\label{nablau2}\end{equation}
where $dg$ denotes the Riemannian measure on $\M$ and $d\sigma$ induced surface measure on $\partial K_r$. Note that the normal derivative $\partial_\nu$ here coincides with minus the radial derivative, $-\partial_r$. 

On the end $E_-$, $u$ is bounded, while, thanks to the asymptotic expansion of $\partial_r u$ from Lemma~\ref{lem:harmonic-}, we have 
\begin{equation}
|\partial_r u| \leq \frac{C}{r^2}
\end{equation}
on the $E_-$ end. This decay beats the growth of the surface measure of $K_r \cap E_-$, which is $O(r)$, so the contribution of this part to the limit in \eqref{nablau2} vanishes in the limit. Similarly, on the $E_+$ end, $u$ has an asymptotic expansion, beginning with the power $r^{2-n_+}$, while $|\partial_\nu u| \leq C r^{1-n_+}$. Since the surface measure of $K_r \cap E_+$ is $O(r^{n_+-1})$, and $n_+ \geq 3$, the contribution from $E_+$ also vanishes in the limit. 
It follows that the surface integral in \eqref{nablau2} decays to zero. We conclude that 
$$
\int_{\M} |\nabla u|^2 \, dg = 0
$$
and hence that $u$ is constant. Since $u$ tends to zero along the $E_+$ end, this constant is zero, proving the lemma.
\end{proof}

An immediate consequence of Lemma~\ref{lem:uniqueness} is the following result. 

\begin{lemma}\label{lem:Laplace-solvability} Let $F \in L^2(\M)$ be supported in the interior of $K$. Then there is a unique global solution $u$ to Laplace's equation $\Delta u = F$ on $\M$ such that $u$ is bounded on the $E_-$ end and tending to zero at the $E_+$ end. Moreover, $u$ tends to a constant $\beta$ as $r \to \infty$ on $E_-$, with an asymptotic expansion on $E_-$ in nonpositive integral powers of $r$, and has an asymptotic expansion on $E_+$ in negative integral powers of of $r$, as in Lemmas ~\ref{lem:harmonic-} and ~\ref{lem:harmonic+}. 
\end{lemma}

\begin{proof} We solve the boundary value problem \eqref{Laplace}, \eqref{Dtn-bc} for $f |_K$ and extend to
the ends as discussed above.
\end{proof}

We can also use the elliptic boundary problem to find a global harmonic function with logarithmic growth at the $E_-$ end:

\begin{lemma}\label{lem:U}
There exists a unique global harmonic function, $\mathcal{U}$, on $\M$ such that
\begin{equation*} 
\mathcal{U}(z) = 
\begin{cases}
 \log r + c_1 + O(r^{-1}) \qquad &\text{as } r \rightarrow \infty \text{ on the } E_- \text{ end} \\
O(r^{-1}) \qquad &\text{as } r \rightarrow \infty \text{ on the } E_+ \text{ end}.
\end{cases}
\end{equation*}
\end{lemma}

\begin{proof}
Choose a smooth cutoff function $\chi(r)$ supported on $E_- \cup K$ such that $\chi(r) = 1$ for $r$ large and is  such that $\supp \nabla \chi$ is contained in the interior of $K$.  Let $w_1 = \chi(r) \log r$. Then $F := -\Delta w_1$ is supported in $K$. Now solve the elliptic boundary problem 
$$
\Delta w_2 = F \text{ on } K, \quad B w_2 = 0 \text{ on } \partial K
$$
and extend $w_2$ to be harmonic outside $K$, bounded on $E_-$ and decaying to zero on $E_+$. Then the function 
$\mathcal{U} = w_1 + w_2$ is globally harmonic and has the given asymptotics. Uniqueness follows as in the proof of Lemma~\ref{lem:uniqueness}. 
\end{proof}


\subsection{The key lemma}

A key lemma in \cite{CCH} stated the following. (We have paraphrased the result; there were more precise regularity statements in \cite{CCH}.)

\begin{lemma}[Lemma 3.2 of \cite{CCH}]
Let $\M$ be a connected sum of Euclidean spaces of dimension $d \geq 3$. Then, for any $v \in C_c^\infty(\M)$, there is an approximate solution $u(z, k)$ of the equation $(\Delta + k^2) u = v$ such that $|u(z, k)| \leq Cr^{2-d} e^{-ckr}$ for some $c> 0$ and $|(\Delta + k^2) u - v| \leq C k r^{-\infty}$. 
\end{lemma}

Of course, the latter statement is shorthand for the statement that for each integer $M$ there is a constant $C_M$ such that $|(\Delta + k^2) u - v| \leq C_M k r^{-M}$. Here $r$ is the radial coordinate on the end of each Euclidean space, extended to be a strictly positive function in the interior of $\M$. 

This was proved in \cite{CCH} by first solving $\Delta u(z, 0) = v$ for a function $u$ decaying as $r^{2-d}$ at infinity, and then extending the function into $\{ k > 0 \}$ by working on a blow-up space. The following result is implicit in the proof of Lemma 3.2 of \cite{CCH}.

\begin{lemma}\label{lem:CCH}
Let $u$ be a function on $\R^d$, $d \geq 2$ such that $u(x) \to 0$ as $|x| \to \infty$ and $v = \Delta u$ is compactly supported. Then  $u(x)$ can be extended to a function $u(x, k)$ for $k \geq 0$, agreeing with $u$ at $k=0$, such that 
\begin{equation}
(\Delta + k^2) u - v = O(k r^{-\infty})
\end{equation}
and satisfying 
\begin{equation}
\begin{cases}  |u(x,k)| = O(r^{2-d} e^{-ckr}), \quad |\nabla u(x,k)| = O(r^{1-d} e^{-ckr}),\quad d \geq 3, \\
|u(x,k)| =  O(r^{-1} e^{-ckr}), \quad |\nabla u(x,k)| =  O(r^{-2} e^{-ckr}),\quad d =2, 
\end{cases}
\label{CCH1}\end{equation}
and
\begin{equation}
\big|\partial_r (u(x, k) - u(x, 0)) \big| = \begin{cases} O(k r^{2-d}) ,\quad d \geq 3, \\
O(k r^{-1}), \quad d =2. 
\end{cases}
\label{CCH-approx}\end{equation}
for $k \leq k_0$ and for some $c > 0$. 
\end{lemma}

This was explicitly shown for $d \geq 3$. However, the result is also true for $d=2$, since then if $u \to 0$ at infinity, this means that $u$ has an expansion at infinity of the form
$$
u \sim \sum_{m \geq 1} \sum_{\pm} a_m r^{-m} e^{\pm im\theta} 
$$
with the constant Fourier mode absent. This avoids the only problem case, namely $d=2$ and $m=0$, in the proof of \cite[Lemma 3.2]{CCH}. We also note that the estimate above is not optimal, but is sufficient for our needs. The optimal estimate for $u$ would be, for $d \geq 3$, that 
$$
|u(x,k)| = O(r^{2-d}), \text{ for } kr \leq 1,  \quad |u(x, k)| = O(k^{d-2} (kr)^{-(d-1)/2} e^{-kr}) \text{ for } kr \geq 1,
$$
but by sacrificing a bit of exponential decay (by taking $c < 1$) we can write the estimate in the simpler form \eqref{CCH1}. 

The next lemma deals with the missing case, when $d=2$ and $m=0$. To state it, we use the following notation, following \cite{HMM}.

\begin{definition}\label{def:ilg} The function $\ilg : [0, 1/2] \to \R$ is the continuous monotone function defined by 
$$
\ilg k := \begin{cases} \frac1{\log(\frac1{k})}, \quad k > 0 \\ 0, \qquad \quad k = 0. \end{cases}
$$
\end{definition}
This function tends to zero as $k \to 0$, but slower than any positive power of $k$. In the following Lemma, the approximate solution $u(x, k)$  will have an expansion in powers of $\ilg k$ rather than in powers of $k$ as in all the previous cases. However, we shall see (after some work) that this really is the `true' behaviour of $u$ when there is an end with Euclidean dimension $2$. In fact, we will show that the resolvent $(\Delta + k^2)^{-1}$ applied to $v$ has just this type of expansion as $k \to 0$ --- see Proposition~\ref{prop:ilgexp}.

\begin{lemma}[The `key lemma'] \label{u2lemma}
Let $v \in C_c^\infty(\M)$, and let $\phi$ be the solution to the equation $\Delta \phi = -v$ guaranteed by Lemma~\ref{lem:Laplace-solvability}. 
Let $r \in C(M)$ denote the distance to a fixed point in the interior of $K$. 
Then for any integer $q$, there exists an approximate solution $u(z, k)$ to the equation $(\Delta + k^2) u = v$,
in the sense that 
\begin{equation}
(\Delta + k^2) u - v = O((\ilg k)^q r^{-\infty}),
\label{approx-soln}\end{equation}
 and such that $u$ is bounded on the $E_-$ end and $O(r^{2-n_+})$ on the $E_+$ end. Moreover, we have estimates
\begin{equation} \label{uestimates}
|u(z,k)| \leq 
\begin{cases}
C   \qquad &z \in K, \\
C r^{2-n_+} \exp \left(-ckr  \right) \,   \qquad &z \in E_+ \\
C \exp \left(-ckr  \right), \qquad &z \in E_- 
\end{cases}
\end{equation}
and
\begin{equation} \label{graduestimates}
|\nabla u(z,k)| \leq 
\begin{cases}
C  \big( r^{-2} + (\ilg k) r^{-1} \big) \exp \left(-ckr  \right)  \qquad &z \in E_- \\
C  r^{1-n_+}  \exp \left(-ckr \right)  \qquad &z \in E_+ \\
C  \qquad &z \in K 
\end{cases}
\end{equation}
for some constant $C$ and $c \in (0,1)$. 

Furthermore, suppose that the limit $\beta$ of the function $\phi$ at infinity along the end $E_-$ (which exists thanks to Lemma~\ref{lem:Laplace-solvability}) is strictly positive. Then the following \emph{lower bound} holds on $E_-$ for $r \geq r_0$ sufficiently large: 
\begin{equation}
 \partial_r \left(u(z,k) + \phi(z)\right)  \geq C \beta \frac{\textup{ ilg} k}{r} 1_{kr \leq \epsilon}   + O(kr^{-1}) \text{ for } r \geq r_0. 
\label{lowerbound}\end{equation}
\end{lemma}

\begin{proof}
We start by setting $u$ equal to the function $-\phi$ at $k=0$ and try to extend to nonzero $k$. 
The function $\phi$ will tend to a constant $\beta$, in general nonzero, at the $E_-$ end, which is exactly the case not covered by Lemma~\ref{lem:CCH}. 
The idea of the construction is to glue together the Bessel function $K_0(kr)$, supported near infinity on the $E_-$ end and which is annihilated by $\Delta + k^2$, with $\phi$. The function $K_0(z)$ has asymptotics
\begin{equation}
K_0(s)  \sim
\begin{cases}
-\log s + c_\gamma + O(s^2 |\log s|) \qquad & s \rightarrow 0, \\
\frac{\sqrt{\pi}}{\sqrt{2}} \frac{e^{-s}}{\sqrt{s}} \qquad & s \rightarrow \infty,
\end{cases}
\label{K0-asympt}\end{equation}
where $c_\gamma = \log 2 - \gamma > 0$, and $\gamma$ is the Euler-Mascheroni constant. Thus, $K_0(kr)$ has a $-\log k$ divergence as $k \to 0$. To match this with $-\phi$ which is equal to $-\beta + O(r^{-1})$ as $r \to \infty$, we multiply $K_0(kr)$ by a factor $-\beta \ilg k$. 

We thus need to extend the function $\phi - \beta$ to positive values of $k$ on the $E_-$ end, and $\phi$ to positive values of $k$ on the $E_+$ end. This can be done using the previous construction of \cite{CCH} thanks to Lemma~\ref{lem:CCH}. In fact, on the end $E_-$, the projection of $\phi - \beta$ onto the constant functions on the cross-section manifold $\M_-$ may be regarded as a function on $\R^2$ (defined near infinity), so Lemma~\ref{lem:CCH} can be used to extend this projection. On the other hand, the projection onto the orthogonal complement of the constant functions on $\M_-$ is already rapidly (indeed, exponentially) decreasing at infinity, according to the proof of Lemma~\ref{lem:harmonic-} (see estimate \eqref{unif-exp-decay} for the $l \geq 1$ terms of \eqref{u-exp-exp}), so this part can be extended trivially (i.e.\ constantly) in $k$. The extension on the $E_+$ end works in exactly the same way. 
We shall denote these extensions by $\V_-$ and $\V_+$ respectively. (The extension is done separately on each end, so there is no difficulty in extending $\phi - \beta$ on the $E_-$ end and $\phi$  on the $E_+$ end.) We shall assume that $\V_-$ extends $\phi - \beta$ on  $E_- \cup K$ and $\V_+$ extends $\phi$ on $E_+ \cup K$. In fact, we may assume that $\V_+(z,k) = \V_-(z,k) + \beta$ when $z \in K$ as the extension can be taken to be trivial (independent of $k$) outside a neighbourhood of infinity. 

With these preparations, we write down an extension that satisfies the conditions of the lemma for $q=1$. 
Define $\chi_-$ to be a smooth cutoff function such that $\chi_- =1$ on $E_-$ and $\chi_-$ is supported on $E_- \cup K$. We may assume that $\chi_-$ is a function of $r$ only. Define
\begin{equation}
u_1(z,k) = -\beta \chi_-(r) \ilg k \; K_0(kr) - \chi_-(r) \V_-(z, k) -  (1-\chi_-(r)) \V_+(z, k) .
\label{u1def}\end{equation}

Let us check the assertions of the lemma. First, we have 
\begin{equation}\begin{gathered}
(\Delta + k^2) u_1(z, k) = 2\beta \ilg k \nabla \chi_- \nabla K_0(kr) - \beta \ilg k (\Delta \chi_-) K_0(kr) \\
- \chi_-(z) (\Delta + k^2) \V_-(z, k) + 2 \nabla \chi_- \nabla \V_-(z, k) - (\Delta \chi_-) \V_-(z, k) \\
- (1-\chi_-) (\Delta + k^2) \V_+(z, k) - 2\nabla \chi_-  \nabla \V_+(z, k)  + \Delta \chi_-  \V_+(z, k).
\end{gathered}\label{u1}\end{equation}

Examining the RHS of \eqref{u1}, we see that $(\Delta + k^2) \V_-(z, k)$ and $(\Delta + k^2) \V_+(z, k)$ are both $O(k r^{-\infty})$ by Lemma~\ref{lem:CCH}, so these are acceptable errors for any $q$. A key cancellation occurs in the combination of $\Delta \chi_-$ terms. Using 
the asymptotic \eqref{K0-asympt} we find that three terms containing $\Delta \chi_-$ can be estimated by  
\begin{multline*}
\Delta \chi_-\bigg(\beta \ilg k (\log k + \log r + O(1)) - \V_- + \V_+ \bigg) \\=  \Delta \chi_-\bigg(-\beta  - \V_- + \V_+\bigg) + O(\ilg k) \log r \Delta \chi_-   = O(\ilg k) \log r\Delta \chi_- . 
\end{multline*}
since $\Phi_+ = \Phi_- + \beta$ on $K$,  which contains $\supp \nabla \chi_-$. Similarly, we have 
 $\nabla \chi_- (\nabla \V_- - \nabla \V_+)=0$. Lastly, $\nabla \chi_- \nabla K_0(kr)$ is a smooth function of $z$ up to $O(k^2 |\log k|) + \ilg k \nabla \chi_- \nabla \log r$ since the gradient kills the $\log k$ divergence (and $r$ is bounded on the support of $\nabla \chi_-$). So we find that 
\begin{equation}
(\Delta + k^2) u_1(z, k) = - \ilg k \,  v_1 + O(k r^{-\infty}), \text{ where } v_1 \in C_c^\infty(\M). 
\end{equation}
The error term $O(k r^{-\infty})$ is acceptable for any $q$. Now we can repeat the construction, replacing $v$ with $v_1$. Let $u_2$ denote the corresponding expression \eqref{u1}.  
Then $u_1 + \ilg k u_2$ satisfies the assertions of the lemma with $q=2$. By repeating the construction $q$ times we obtain an error of the form $(\ilg k)^q v_q + O(k r^{-\infty})$ with $v_q \in C_c^\infty(\M)$. 

To prove estimate \eqref{uestimates}, it is enough to do so for $u_1$, as the other summands for $u$ are similar with an additional factor of a positive power of $\ilg k$. 
Examining \eqref{u1def}, we check that each term on the RHS satisfies the estimate separately. In the case of the $K_0$ term, this follows from asymptotics \eqref{K0-asympt} and the fact that, when $kr \le 1$, $k$ is small and $r$ is large, $|\log kr| \leq |\log k|$. In the case of the other two terms, the estimate is direct from Lemma~\ref{lem:CCH}. 

Now we check estimate \eqref{graduestimates}. Writing $u = u_1 + (\ilg k) u_2 + \dots + (\ilg k)^{q-1} u_q$, we see using \eqref{K0-asympt} and \eqref{CCH1} that all terms $(\ilg k)^{p-1} u_p$ satisfy the estimate on the RHS. So it is enough to check the estimates when $u = u_1$. From \eqref{u1def}, we have 
\begin{multline}\label{gradu1expr}
\nabla u_1(z,k) = - \nabla \chi_- \Big(  \beta \ilg k \, K_0(kr)  + \V_-(z, k) - \V_+(z, k)  \Big) \\
- \beta  \chi_-(z)  \ilg k \nabla  K_0(kr) - \chi_-(z) \nabla \V_-(z, k)  -  (1-\chi_-(r)) \nabla \V_+(z, k).
\end{multline}
The first line is supported on $K \times [0, k_0]$, and it is clear that the first term is $O(1)$. 
In the second line, the terms involving $\nabla \V_\pm$ satisfy the estimate \eqref{CCH1}, while $\ilg k\nabla   K_0(kr)$ on $\supp \chi_-$ is bounded by $(\ilg k )\, r^{-1}$ for $kr \leq 1$ or $r^{-1} e^{-ckr}$ for $kr \geq 1$ using the asymptotics at zero, respectively infinity, for the modified Bessel function $K_0$. This verifies \eqref{graduestimates}. 

To check \eqref{lowerbound}, we note that it is only necessary to consider $u_1$, since for $q \geq 2$, $\nabla u_q = O((\ilg k)^q r^{-1}) + O((\ilg k)^{q-1} r^{-2})$ for large $r$ on $E_-$. We write, using \eqref{gradu1expr}, 
\begin{multline}\label{graduphi1expr}
\partial_r \big( u_1(z,k) + \phi(z) \big)  = - \partial_r  \chi_- \Big( \beta \ilg k \, K_0(kr)  + \V_-(z, k) - \V_+(z, k)  \Big) \\
- \beta  \chi_-(z)  \ilg k \,  \partial_r   K_0(kr) - \chi_-(z) \big( \partial_r  \V_-(z, k) - \partial_r  \phi  \big)  -  (1-\chi_-(r)) \big( \partial_r  \V_+(z, k) - \partial_r  \phi   \big).
\end{multline}
By choosing $r_0$ large enough, the first term is zero on the support of the characteristic function in \eqref{lowerbound}. Also, the quantity $$\chi_-(z) \big( \partial_r  \V_-(z, k) - \partial_r  \phi \big)  +  (1-\chi_-(r)) \big( \partial_r  \V_+(z, k) - \partial_r  \phi \big)$$ is $O(k r^{-1})$, by \eqref{CCH-approx}. 
Finally, the remaining term $- \beta  \chi_-(z)  \ilg k \, \partial_r   K_0(kr)$ is positive, since $K_0$ is strictly decreasing. Using the asymptotic expansion of $K_0$ as its argument tends to zero, we find that this is equal to $\beta \ilg k \, r^{-1} (1 + O(\epsilon^2 \log \epsilon))$ for $kr \leq \epsilon$, which has a lower bound of $\beta \ilg k \, r^{-1}/2$ for $\epsilon$ sufficiently small.

\end{proof}

\begin{remark}
It is instructive to compute the coefficient of $\ilg k$ in the expansion of $u(z, k)$ at $k=0$. This coefficient arises from both the $u_1$ and $u_2$ terms. The contribution of  $u_1$ to the coefficient of $\ilg k$ of $u$,  using asymptotics \eqref{K0-asympt}, is
$$
-\beta \chi_- (-\log r + c_\gamma) .
$$
Denote by $\phi_1$ the solution to $\Delta \phi_1 = -v_1$ given by Lemma~\ref{lem:Laplace-solvability}, $\beta_1$  the limit of $\phi_1$ on the $E_-$ end, and by $\V_{1, -}$ the extensions of $\phi_1 - \beta_1$ on the end $E_-$ and $\V_{1, +}$ the extension of $\phi_1$ on the end $E_+$. The contribution of $u_2$ to the coefficient of $\ilg k$ in the expansion of $u$ is 
$$
\Big( - \beta_1 \chi_- -\chi_- \V_{1,-} - (1 - \chi_-) \V_{1, +} \Big) \Big|_{k=0} = -\chi_- \phi_1 - (1 - \chi_-) \phi_1 = -\phi_1 . 
$$
Note that explicitly, $v_1$ is given by
$$
v_1 = -2\beta \nabla \chi_- \nabla (-\log r + c_\gamma) + \beta \Delta \chi_- (-\log r + c_\gamma) =  \beta\Delta ( \chi_- (-\log r + c_\gamma)),
$$
since $\Delta (\log r) = 0$ on the support of $\chi_-$. Using the fact that $v_1 = - \Delta \phi_1$, the overall coefficient of $\ilg k$ in the expansion of $u$ can therefore be expressed as 
\begin{equation}\label{ilgk-coeff}
\beta \chi_- (-\log r + c_\gamma) -   \Delta^{-1} \Big( \Delta \big( \beta \chi_- (-\log r + c_\gamma) \big) \Big),
\end{equation}
 where $\Delta^{-1}$ is given by Lemma~\ref{lem:Laplace-solvability}. (Here we notice that since $\phi_1$ solves $\Delta \phi_1 = - v_1$ and $\phi$ tends to a constant at the $E_-$ end and tends to zero at the $E_+$ end, $\phi_1$ is indeed given by $\Delta^{-1} (-v_1)$. )
 
It is clear that \eqref{ilgk-coeff} is a global harmonic function on $\mathcal{M}$. Moreover, it is \emph{not} trivial, despite initial appearances, since the expression $\beta \chi_- (-\log r + c_\gamma)$ blows up logarithmically on $E_-$ while the second term, which is exactly $-\phi_1$, has limit $-\beta_1$ at the $E_-$ end. 
Thus, \eqref{ilgk-coeff} has asymptotic $-\beta \log r - \beta_1$ at the $E_-$ end, and is therefore equal to $-\beta \mathcal{U}$ where $\mathcal{U}$ is  from Lemma~\ref{lem:U}. 
\end{remark}

\begin{remark}
There is a slight improvement that can be made in the estimate \eqref{graduestimates} when $v$ is equal to $\Delta \phi$, where $\phi$ vanishes on the $E_+$ end (which is a situation that occurs in the parametrix construction of the next section). In this case, the function $\V_+$ can be taken to vanish on the $E_+$ end, and this allows us to improve the estimate as follows:
\begin{equation} \label{graduestimates2}
|\nabla u(z,k)| \leq 
\begin{cases}
C  \big( r^{-2} + (\ilg k) r^{-1} \big) \exp \left(-ckr  \right)  \qquad &z \in E_- \\
C  \ilg k \, r^{1-n_+}  \exp \left(-ckr \right)  \qquad &z \in E_+ \\
C  \qquad &z \in K .
\end{cases}
\end{equation}
The key point is the gain of a factor of $\ilg k$ along $E_+$.  We will use this improved estimate in Section~\ref{sec:le}. 
\end{remark}




\section{Parametrix construction}

In this section we construct a parametrix, $G(k)$, for the resolvent, $(\Delta + k^2)^{-1}$, on $\M$. It is useful to define our parametrix via its Schwartz kernel. First define $\phi_\pm \in C^\infty(\M)$ satisfying the properties: (i) $\supp \phi_\pm$ is contained in $E_\pm$, (ii) $0 \leq \phi_\pm \leq 1$, (iii) $\phi_\pm$ equals 1 near infinity on the $E_\pm$ end. Because of property (i), we may view $\phi_\pm$ as defined on $\R^{n_\pm} \times \M_\pm$ as well as on $E_\pm$. We define $v_\pm$ by
\begin{equation}
v_\pm = -\Delta \phi_\pm  \in C_c^\infty(\M).
\end{equation}
We construct an approximate solution $u_\pm$ to the equation 
\begin{equation}
(\Delta + k^2)u_\pm \simeq v_\pm
\end{equation}
using Lemma~\ref{u2lemma} with $q=2$. 

Next, choose an open set $O$ containing $K$ so that $\phi_- + \phi_+ = 1$ in a neighbourhood of the complement $O^c$ of $O$. Then 
let $G_\text{int}(k)(z,z')$ be an interior parametrix for the resolvent, with kernel compactly supported in $\M \times \M$, such that $(\Delta + k^2) G_\text{int}(z,z') - \delta_{z}(z')$ is a smooth function for $z \in O$. 

Let $z_\pm^\circ \in \R^{n_\pm} \times \M_\pm$ be a point such that $z_\pm^\circ \notin \supp\phi_\pm$. Also, to streamline the notation, we abbreviate $\Delta_{\R^{2}\times \M_-}$ to $\Delta_-$ and $\Delta_{\R^{n_+}\times \M_+}$ to $\Delta_+$.

We define our `pre-parametrix' $\tG(k) = G_1(k) + G_2(k) + G_3(k)$ (which will be modified by the addition of $G_4(k)$ below to form our actual parametrix) where
\begin{eqnarray*}
G_1(k)(z,z') &=& (\Delta_- + k^2)^{-1}(z,z')\phi_-(z) \phi_-(z') + (\Delta_+ + k^2)^{-1}(z,z')\phi_+(z) \phi_+(z') \\
G_2(k)(z,z') &=& G_{\text{int}}(k)(z,z') \left(1 - \phi_-(z) \phi_-(z') - \phi_+(z) \phi_+(z')\right)\\
G_3(k)(z,z') &=& (\Delta_- + k^2)^{-1}(z_-^\circ,z') u_-(z,k) \phi_-(z') + (\Delta_+ + k^2)^{-1}(z_+^\circ,z') u_+(z,k) \phi_+(z').
\end{eqnarray*}

We define the error term $\tE(k)$ by
\begin{equation*}
(\Delta + k^2) \tG(k) = \Id + \tE(k).
\end{equation*}
The error can be calculated as
\begin{align}\label{tE}
\begin{split} 
\tE(k)(z,z') ={}& \ttE(k)(z,z') + \tttE(k)(z,z'), 
\end{split} \\
\begin{split}
\ttE(k)(z,z') ={}& 
-2 \sum_\pm \nabla \phi_\pm(z) \phi_\pm(z') \left(\nabla(\Delta_{\pm} + k^2)^{-1}(z,z') - \nabla G_{\text{int}}(k)(z,z') \right) \nonumber
\end{split} \\
\begin{split}
+{} &\left((\Delta+k^2)G_{\text{int}}(k)(z,z') - \delta_{z'}(z) \right)\left(1 - \sum_\pm \phi_\pm(z) \phi_\pm(z') \right) \nonumber
\end{split} \\
\begin{split}
+{} \sum_\pm &v_\pm(z) \phi_\pm(z')\Big((\Delta_{\pm}+k^2)^{-1}(z_\pm^\circ,z') -(\Delta_{\pm}+k^2)^{-1}(z,z') + G_{\text{int}}(k)(z,z')\Big) \nonumber
\end{split} \\
\begin{split}
\tttE(k)(z,z') ={}& \sum_\pm \phi_\pm(z')[(\Delta + k^2)u_\pm(z,k) - v_\pm(z)] (\Delta_{\pm}+k^2)^{-1}(z_\pm^\circ,z')  .
\end{split} \nonumber
\end{align}
It is straightforward to check that $\tE(k,z,z')$ is smooth in each variable, and compactly supported in the left variable $z$. 
Let $\chi(z)$ be a function that is $1$ on the support of $\nabla \phi$, with compact support. 
Estimating the  difference in resolvent kernels in the third line by using the gradient estimate of the resolvent kernel in \eqref{gradresolvent-2}, we have an estimate 
\begin{equation}
|\ttE(k,z,z')| \leq 
\begin{cases}
C \chi(z), \qquad &z' \in K \\
C \chi(z) {r'}^{1-n_+} \exp \left(-ckr' \right) \qquad &z' \in E_+ \\
C \chi(z) {r'}^{-1} \exp \left(-ckr' \right)  \qquad &z' \in E_-.
\end{cases}
\label{ttE}\end{equation}
On the other hand, using \eqref{approx-soln} and \eqref{resolvent-3}, \eqref{resolvent-2}, we obtain 
\begin{equation}
|\tttE(k,z,z')| \leq 
\begin{cases}
C (\ilg k)^{q-1} \, r^{-\infty} , \qquad &z' \in K \\
C (\ilg k)^q  \, r^{-\infty}  {r'}^{2-n_+} \exp \left(-ckr' \right) \qquad &z' \in E_+ \\
C (\ilg k)^q \, r^{-\infty}  (1+ |\log (kr')|) \exp \left(-ckr' \right)  \qquad &z' \in E_-.
\end{cases}
\label{tttE}\end{equation}

\subsection{Compactness of the error term}
Following the strategy in \cite{CCH} or Part I, we show that the error term $\tE(k)$ is compact. In Part I,  for dimensions $n_\pm \geq 3$,  the error was Hilbert-Schmidt on $L^2(\M)$. However, in the present case, this is not true. There are two reasons for this: first, the gradient of the resolvent in the $\ttE(k)$ term decays as ${r'}^{-1}$ as the right variable tends to infinity on the $E_-$ end, which is not square-integrable, unlike the case for all $n \geq 3$. More seriously, the term $\tttE(k)$ only decays as ${r'}^{2-n_+}$ as $r' \to \infty$ on the $E_+$ end (which is not $L^2$ for $n = 3$ or $4$) and, even worse, only as $(\ilg k)^q \log kr' e^{-ckr'}$ as $r' \to \infty$ on the $E_-$ end, where we chose $q=2$ above. On the other hand, $\tE(k)$ is rapidly decreasing as the \emph{left} variable tends to infinity. 
To retain the Hilbert-Schmidt property, we work on a weighted Hilbert space $L^2_w$
for some suitably decaying weight $w$, where
\begin{equation*}
L^2_w(\M) = \{ f \in L^2_{loc}(\M) \mid w^{-1} f \in L^2(\M) \}. 
\end{equation*}
This allows the kernel of the error to be in the dual space, $L^2_{w^{-1}}$ (hence allowing some growth) in the right (primed) variable. So $w$ must be chosen to decay at zero at infinity fast enough so that the kernel of the error is in $L^2_w(\M) \otimes L^2_{w^{-1}}(\M)$. 

The situation is however, quite delicate, because we also need to ensure the validity of Lemma~\ref{lem:density}. The proof of Lemma~\ref{lem:density} relies on the nonexistence of a nontrivial harmonic function in the dual space $L^2_{w^{-1}}$. On the other hand, in Lemma~\ref{lem:U} we constructed a 
harmonic function $\mathcal{U}$ that is asymptotic to $\log r$ as $r \to \infty$. Thus we require a weight such that $\log r$ is not in $L^2_{w^{-1}}$, yet $(\ilg k)^q (1 + |\log (kr')|) e^{-ckr'}$ is in this space $L^2_{w^{-1}}$ uniformly as $k \to 0$. 

We claim that, with $q=2$, the weight function $w$ given by
\begin{equation}
w(z) = \begin{cases}
1, \quad z \in K \\
r^{-1}, \quad z \in E_+ \\
(r \log r)^{-1}, \quad z \in E_-
\end{cases}
\label{w}\end{equation}
is a suitable weight. It is then clear that $\mathcal{U} \notin L^2_{w^{-1}} (\M)$, since $r^{-1}$ is not $L^2$ on the end $E_-$. On the other hand, we have, integrating on the $E_-$ end only (a similar but simpler calculation holds on the $E_+$ end), if $f(z) = (1+ |\log (kr')|) \exp \left(-ckr' \right)$, then 
\begin{equation}\begin{gathered}
\| f \|^2_{L^2_{w^{-1}}(E_-)} \leq C \int_R^\infty (1 + |\log (kr)|)^2 e^{-2ckr} (r \log r)^{-2} \, r dr \\
\leq C \Big( \int_{Rk}^1 (1 + |\log s|)^2 s^{-1} ds +  \int_1^\infty e^{-2cs} s^{-1} ds \Big)  \\
\leq	 C |\log k|^3.
\end{gathered}\label{k-blowup}\end{equation}
Now using the factor $(\ilg k)^2$ on the end $E_-$ from \eqref{tttE}, we see that, if we have $q=2$, then $(\ilg k)^q f$ tends to zero in $L^2_{w^{-1}}(E_-)$ as $k \to 0$, but this is not true if $q=1$.

\begin{lemma}\label{lem:HS} The error term $\tE(k)$ is a Hilbert-Schmidt operator on $L^2_w(\M)$, uniformly for $0 < k \leq k_0$. Moreover, as an element of $\HS(L^2_w(\M))$, the Hilbert-Schmidt operators on $L^2_w(\M)$, it is continuous in $k$ and has a limit $\tE(0)$ as $k \to 0$. 
\end{lemma}

\begin{proof} The kernel $\ttE(k, z, z')$ is smooth in $(z, z')$ and has a limit as $k \to 0$. To prove continuity in the space $\HS(L^2_w(\M)) = L^2_w(\M) \otimes L^2_{w^{-1}}(\M)$, it suffices to check that it is dominated pointwise by a $k$-independent kernel in $\HS(L^2_w(\M))$, as then continuity follows from the Dominated Convergence Theorem. A suitable dominating function is furnished by taking \eqref{ttE} and removing the $k$-dependent factors of $e^{-ckr'}$. 

On the other hand, the kernel $\tttE(k, z, z')$ is clearly continuous in $k$ with values in $\HS(L^2_w(\M))$ for $k > 0$, since then we have rapid decay at each end. The computation \eqref{k-blowup} above shows that the Hilbert-Schmidt norm of $\tttE(k)$ is $O((\ilg k)^{1/2})$ as $k \to 0$, which shows continuity at $k=0$ as well. 
\end{proof}

\begin{remark} If we integrate just from $r=R$ to $r=R+1$ in \eqref{k-blowup}, we see that this term is bounded below by a multiple of $|\log k|^2$ as $k \to 0$, and this conclusion holds regardless of the weight function $w$. We see from this that it is essential to take $q > 1$ in Lemma~\ref{u2lemma}, otherwise the Hilbert-Schmidt norm of the kernel $\tttE(k, z, z')$ would not tend to zero (or even have a limit) as $k \to 0$.

 This need to take $q > 1$ is not surprising from a different point of view: the key Lemma 3.2 of \cite{CCH} in effect constructed a nontrivial bounded harmonic function globally on the manifold under consideration there, while when $q=2$ (but not when $q=1$), the analogous Lemma~\ref{u2lemma} in the present paper constructs the global harmonic function $\mathcal{U}$ which is `almost bounded' (logarithmic growth at the $E_-$ end and bounded at every other end). 
\end{remark}

\subsection{Finite rank perturbation}
Now that we have shown that $\tE(k)$ is Hilbert-Schmidt, and a fortiori compact, on a weighted Hilbert space, we wish to invert $\Id + \tE(k)$ to obtain the true resolvent. Clearly this operator is Fredholm with index zero.  However, it may not be invertible, i.e. it may have a non-trivial null space. In this section we perturb the parametrix $G$ by a finite rank operator, independent of $k$, in order to make $\Id + E(0)$ invertible. Using the continuity proved in Lemma~\ref{lem:HS}, we automatically obtain invertibility of $\Id + \tE(k)$ for $k$ sufficiently small.

Let $\varphi_1,\ldots,\varphi_M$ be a basis for the null space of $\Id + \tE(0)$. Note $\varphi_1,\ldots,\varphi_M$ are smooth and compactly supported in the left variable since the Schwartz kernel of $\Id + \tE(0)$ is smooth and compactly supported in the left variable (noting that $\tE(0) = \tE'(0)$). Next we would like to choose functions 
$\psi_1,\ldots, \psi_M \in C_c^\infty(\M)$ such that $\Delta \psi_1,\ldots, \Delta\psi_M$ span a space supplementary to the range of $\Id + \tE(0)$. (Note since $\Id + \tE(0)$ is Fredholm of index zero, the null space and cokernel have the same dimension~$M$.) To show that these functions exist we need the following results.

\begin{lemma}\label{lem:density}
Let $\Delta$ be the Laplacian on $\M$ and let $w$ be the weight function \eqref{w}. The range of $\Delta$ acting on $C_c^\infty(\M)$ is dense in $L^2_w(M)$. 
\end{lemma}

\begin{proof}
Denote $\Delta$ acting on $C_c^\infty(\M)$ by $\Delta|_{C_c^\infty}$. Let $u \in L^2_{w^{-1}}(\M)$, viewed as the dual space of $L^2_w(\M)$, be such that $u$ annihilates the range of $\Delta|_{C_c^\infty}$. To prove the lemma, it suffices to show that $u$ vanishes identically.

Such a function $u$ is in the domain of the adjoint $\Delta^*$ of $\Delta|_{C_c^\infty}$, and satisfies $\Delta^* u = 0$. It follows from formal self-adjointness that $\Delta u = 0$ in the distributional sense. By Weyl's lemma for harmonic functions, it follows that $\Delta u = 0$ also in the classical sense and that $u$ is smooth on  $\M$. 

Now on the end $E_-$, using coordinates $(r, \theta, y)$ with $(r, \theta)$ standard polar coordinates on $\R^2$ and $y$ a coordinate on $\M_2$, we expand $u$ in the eigenfunctions $\psi_n$ on $\M_2$ as well as in a Fourier series in $\theta$, as in \eqref{u-exp-2}. 
The same reasoning as in the proof of Lemma~\ref{lem:harmonic-} shows that the coefficients $b_{mn}(r)$ are either modified Bessel functions, for $n > 0$, powers $\pm m$, for $n= 0$ and $m \neq 0$, or a linear combination of constant and logarithm for $n=m=0$. 
Next we apply the condition that $u$ is in $L^2_{w^{-1}}(E_-)$. Notice that $w$ is a function only of $r$. So the expansion \eqref{u-exp-2} for $u$ is orthogonal in this weighted space, and by Bessel's inequality, each term in the sum must separately lie in $L^2_{w^{-1}}$. This condition means that we must take the exponentially decaying modified Bessel functions for $n > 0$, the decaying powers for $n=0$ and $m > 0$, and, most crucially, the constant (but not the logarithm) in the case $n=m=0$, since the logarithm barely fails to lie in our weighted $L^2$ space. That is, $u$ is in fact a bounded harmonic function on the $E_-$ end. 

Exactly the same reasoning applies to the $E_+$ end, and shows that $u$ is a bounded harmonic function decaying to zero on the $E_+$ end. 

Together, these statements show that $u |_K$ solves the homogeneous elliptic boundary problem \eqref{Laplace}, \eqref{Dtn-bc}, with $F=0$.  Lemma~\ref{lem:uniqueness} then implies that $u$ is the zero function, completing the proof. 
\end{proof}

It follows that
\begin{lemma}\label{lem:supp}
There exists functions $\varphi_1,\ldots,\varphi_M \in C_c^\infty(\M)$ and $\psi_1,\ldots, \psi_M \in C_c^\infty(X)$ such that $\varphi_1,\ldots,\varphi_M$ are a basis of  the null space of $\Id + \tE(0)$ and $\Delta \psi_1,\ldots,\Delta \psi_M$ are a basis for a subspace supplementary to the range of $\Id + \tE(0)$.
\end{lemma}

Define the rank $N$ operator $G_4$ by
\begin{equation}
G_4 = \sum_{i=1}^M \psi_i \langle \varphi_i, \cdot \rangle.
\label{G4}\end{equation}
where $\langle \varphi_i, \cdot \rangle$ denotes inner product with $\varphi_i$.
We define $G(k) = \tG(k) + G_4$ and $E(k) = (\Delta+k^2)G(k) - \Id = \tE(k) + (\Delta+k^2)G_4$. 
We then have

\begin{lemma}\label{lem:inv}
The operator $\Id  + E(0) = \Id + \tE(0) + \Delta G_4$ is invertible.
\end{lemma}

\begin{proof}
Indeed, if $(\Id + \tE(0) + \Delta G_4) h = 0$
then $\Delta (G_4 h) = - (\Id + \tE(0)) h$ is simultaneously in the range of $\Id + \tE(0)$ and in the span of $\Delta \psi_i$. 
Therefore, $\Delta (G_4 h) = - (\Id + \tE(0)) h = 0$ using the supplementary property of Lemma~\ref{lem:supp}. This means
that $h$ is in the null space of $\Id + \tE(0)$, hence a linear combination of the $\varphi_i$. But then, assuming without loss of generality that the $\varphi_i$ are orthogonal in $L^2(\M)$, the equation $\Delta G_4 h = 0$ and \eqref{G4} 
implies that $h=0$. So the operator $\Id + \tE(0) + \Delta G_4$ has trivial null space. Since $\tE(0) + \Delta G_4$ is compact, this implies invertibility. 
\end{proof}


\subsection{Inverting $\Id + E(k)$}

We have $(\Delta + k^2) G(k) = \Id + E(k)$ By Lemma~\ref{lem:inv}, $\Id + E(0)$ is invertible, and by continuity (Lemma~\ref{lem:HS}), this implies that $\Id + E(k)$ is invertible for sufficiently small $k$, say $k \leq k_0$. We can therefore obtain the exact resolvent in the form 
$$
R(k) := (\Delta+k^2)^{-1} = G(k)(\text{Id} + E(k))^{-1}. 
$$
Write the inverse as
\begin{equation} \label{inverseF}
(\text{Id} + E(k))^{-1} = \text{Id} + S(k),
\end{equation}
so that
\begin{equation} \label{RGS}
R(k) = G(k)(\Id + S(k)) = G(k) + G(k) S(k).
\end{equation}

Our next task is to show that  $S(k)$ has similar properties as $E(k)$, in particular satisfies similar pointwise estimates. Clearly, $S(k)$ is a bounded
operator on $L^2_w(\M)$. Next, we derive an expression for $S(k)$ as follows. Note that from $(\ref{inverseF})$ we get
\begin{eqnarray}
\text{Id} = (\text{Id} + E(k))(\text{Id} + S(k)) = (\text{Id} + S(k))(\text{Id} + E(k)) 
\end{eqnarray}
It follows that 
\begin{equation*}
S(k) = -E(k) - E(k)S(k) = -E(k)-S(k)E(k).
\end{equation*}
Iterating gives
\begin{equation}
S(k) = -E(k) + E(k)^2 + E(k)S(k)E(k).
\label{Sk-identity}\end{equation}

Since $S(k)$ is a bounded operator, and $E(k)$ is a Hilbert-Schmidt operator, on $L^2_w$, this equation shows that $S(k)$ is also Hilbert-Schmidt, that is,
we have 
\begin{equation}
S(k) \in L^2_w(\M) \otimes L^2_{w^{-1}}(\M). 
\label{Sk-est1}\end{equation}
Moreover, we see from the explicit representation \eqref{tE} of $\tE(k)$, as well as the explicit representation of $(\Delta + k^2) G_4$, that we have
$E(k)$ in the spaces 
\begin{equation}
E(k) \in r^{-\infty} L^\infty(\M) \otimes L^2_{w^{-1}}(\M), \quad E(k) \in L^2_w(\M) \otimes L^\infty_{\mu'}(\M),
\label{Ek-est}\end{equation}
as well as 
\begin{equation}
E(k) \in r^{-\infty} L^\infty(\M) \otimes  L^\infty_{\mu'}(\M),
\label{Ek-est2}\end{equation}
where $L^\infty_{\mu'}(\M)$ is a $k$-dependent weighted $L^\infty$ space given by 
\begin{equation}
L^\infty_{\mu'}(\M) = \{ f \in L^\infty_{loc}(\M) \mid {\mu'}^{-1}f \in L^\infty(\M) \},
\end{equation}
\begin{equation}
\mu'(z', k) = \begin{cases}  1 \quad z' \in K \\
 {r'}^{2-n_+}  e^{-ckr'} \quad z' \in E_+ \\
 e^{-ckr'} \quad z' \in E_-.
\end{cases}
\end{equation}
To check the last line, note that if $kr \geq 1$ is large then the $\log kr$ factor from \eqref{tttE} can be absorbed in the exponential decay by adjusting the value of $c$, while for
$kr \leq 1$ we have $|\log kr| \leq |\log k|$, so this factor is then cancelled by one factor of $\ilg k$. (These estimates are not optimal, but they suffice for proving Riesz transform boundedness for $1 < p < 2$.) 

Now notice that given representations \eqref{Ek-est}, composition gives 
$$
E(k)^2 \in r^{-\infty} L^\infty(\M) \otimes  L^\infty_{\mu'}(\M)
$$
and using also \eqref{Sk-est1}, we have 
$$
E(k) S(k) E(k) \in r^{-\infty} L^\infty(\M) \otimes  L^\infty_{\mu'}(\M).
$$
Thus \eqref{Sk-identity} and \eqref{Ek-est2} shows that \eqref{Sk-est1} improves to 
\begin{equation}
S(k) \in r^{-\infty} L^\infty(\M) \otimes  L^\infty_{\mu'}(\M). 
\label{Sk-est2}\end{equation}

\subsection{Correction to the exact resolvent}

Recall from $(\ref{RGS})$ that we can write the resolvent as
\begin{equation*}
R(k) = G(k) + G(k)S(k).
\end{equation*}
We need to analyse the behaviour of $G(k)S(k)$. This is done by examining each term $G_i(k)S(k)$ for $1 \leq i \leq 4$. Note that  $S(k,z,z')$ is supported in the region $\{z \in \supp \chi \}$, only regions $\{z' \in \supp \chi \}$ are relevant for the kernel $G(k)$. To state the next proposition we define weight functions 
\begin{equation*} 
\begin{gathered}
\mu(z, k) = \begin{cases} 1 + |\log k| , \quad \qquad \quad \qquad  z \in K \\
e^{-ckr} r^{2-n_+}, \quad  \qquad  \quad \qquad  z \in E_+ \\
e^{-ckr} \big( 1 + |\log k| + \log r \big), \   z \in E_-
\end{cases} \end{gathered}
\quad 
\begin{gathered}
\nu(z, k) = \begin{cases} 1, \quad \qquad \qquad z \in K \\
r^{1-n_+} e^{-ckr} , \quad z \in E_+ \\
r^{-1} e^{-ckr} , \quad \quad  z \in E_-. 
\end{cases} \end{gathered}
\end{equation*}

\begin{proposition}
For $k \leq k_0$, the resolvent is given by
\begin{equation*}
(\Delta + k^2)^{-1} = G(k) + G(k)S(k),
\end{equation*}
satisfying 
\begin{equation}\label{GkSk-est}
\Big| (G(k)S(k))(z,z') \Big| \leq C \mu(z,k) {\mu'}(z', k). 
\end{equation}
Furthermore, the gradient of the correction term satisfies the estimate
\begin{equation}
\Big| (\nabla G(k)S(k))(z,z') \Big| \leq C \nu(z,k) {\mu'}(z', k). 
\label{nablaGS-est}\end{equation}
\end{proposition}

\begin{proof}
\underline{$G_1 S$ term:} We claim that this kernel is bounded pointwise by 
\begin{equation*}
\Big| (G_1(k)S(k))(z,z') \Big|  \leq \mu(z, k){\mu'}(z', k). 
\end{equation*}
To prove this estimate, we write the kernel of the composition $G_1(k) S(k)$ as 
\begin{equation*}
G_1(k)S(k)(k,z,z') = \sum_\pm \phi_\pm(z) (\Delta_{\pm} + k^2)^{-1} \left(\phi_\pm(z') S(k,\cdot,z') \right)(z),
\end{equation*}
apply the resolvent estimates \eqref{resolvent-3}, \eqref{resolvent-2} and the estimate \eqref{Sk-est2} for $S(k)$, and integrate in the inner variable. 

We can also use the gradient resolvent estimates \eqref{gradresolvent-3}, \eqref{gradresolvent-2} and obtain the estimate 
\begin{equation*}
|\nabla G_1(k)S(k)(k,z,z')| \leq C \nu(z, k) {\mu'}(z', k). 
\end{equation*}

\underline{$G_2 S$ term:} The kernel $G_2(k,z,z')$ is pseudodifferential of order $-2$, with compact support in $z$, and with no blowup as $k \to 0$. So 
with $\chi$ as in \eqref{ttE}, we get 
\begin{equation}
\big| G_2(k)S(k)(z,z') \big| + \big| \nabla G_2(k)S(k)(z,z') \big| \leq C \chi(z) {\mu'}(z', k).
\end{equation}

\underline{$G_3 S$ term:} This is very similar to the $G_1$ case. Using estimates \eqref{uestimates} and \eqref{graduestimates}, and taking into account the $\log k$ divergence in the $(\Delta_- + k^2)^{-1}(z_-^\circ,z') $ factor, we obtain 
\begin{equation*}
\Big| (G_3(k)S(k))(z,z') \Big|  \leq \mu(z, k) {\mu'}(z', k), \quad \Big| (\nabla G_3(k)S(k))(z,z') \Big|  \leq \nu(z, k) {\mu'}(z', k). 
\end{equation*}
%
%

\underline{$G_4 S$ term:} $G_4$ is a finite rank operator with kernel independent of $k$ and having compact support in $z$. So for a suitably chosen $\chi \in C_c^\infty(\M)$ we have 
\begin{equation}
\big| (G_4S(k))(z,z') \big| + \big| (\nabla G_4S(k))(z,z') \big| \leq C \chi(z) \mu'(z', k).
\end{equation}
\end{proof}

\subsection{Low energy asymptotics of the resolvent}

Using the construction of the resolvent in this section, together with the key lemma of the previous section, we show

\begin{proposition}\label{prop:ilgexp}
Let $\M$ be as above, and let $v \in C_c^\infty(\M)$. Define
$$
U(\cdot, k)=  (\Delta + k^2)^{-1} v. 
$$
Then on any compact set $K' \subset \M$, $U$ has a complete asymptotic expansion in nonnegative integral powers of $\ilg k$, with coefficients that are smooth on $K'$. 
\end{proposition}

\begin{proof}
We first choose an arbitrary positive integer $q$, and use the key lemma,  Lemma~\ref{u2lemma},  to construct an approximate solution $u$. By construction, $u$ has an expansion on $K'$ in nonnegative integral powers of $\ilg k$ up to $(\ilg k)^{q}$. 
Then we have 
$$
U = u - (\Delta + k^2)^{-1} \Big( (\Delta + k^2) u - v \Big).
$$
By Lemma~\ref{u2lemma}, $(\Delta + k^2) u - v = O((\ilg k)^q r^{-\infty})$. Moreover, the explicit representation of $G(k) = G_1(k) + G_2(k) + G_3(k) + G_4(k)$, together with the estimate \eqref{GkSk-est} above for $G(k) S(k)$, shows that the resolvent kernel is $O(|\log k|) = O((\ilg k)^{-1})$. It follows that $U$ has an expansion in powers of $\ilg k$ up to $(\ilg k)^{q-1}$. Since $q$ is arbitrary, this shows that $U$ has in fact a complete asymptotic expansion in powers of $\ilg k$. 
\end{proof}




\section{$L^p$ boundedness of the Riesz transform}\label{sec:le}
The goal of this chapter is to investigate the $L^p$-boundedness of the Riesz transform via the resolvent constructed in the previous section. We use the identity 
\begin{equation}
\Delta^{-1/2} = \frac{2}{\pi} \int_0^\infty (\Delta + k^2)^{-1} \, dk
\end{equation}
We use this representation to break up the Riesz transform into two parts, following \cite{CCH, HS2}. We define 
\begin{equation}\begin{aligned}
F_<(\xi) &= \frac{2}{\pi} \int_0^{1} (\xi^2 + k^2)^{-1} \dk \\
F_>(\xi) &= \frac{2}{\pi} \int_{1}^\infty (\xi^2 + k^2)^{-1} \dk,
\end{aligned}\label{F<>}\end{equation}
Then we call $\nabla F_<(\Delta)$ and $\nabla F_>(\Delta)$ the low-energy and high-energy Riesz transform, respectively. Clearly, the sum of these two operators is the Riesz transform $T = \nabla \Delta^{-1/2}$. So, to prove $L^p$ boundedness of the Riesz transform, it suffices to prove the $L^p$ boundedness of the low-energy and high-energy Riesz transforms separately. 

\subsection{Preparatory results}

The goal of this section to show $\nabla F_<(\sqrt{\Delta})$ is $L^p$-bounded for all $1 < p \leq 2$. 
Recall from Part I that
\begin{equation*}
F_<(\xi) = \frac{2}{\pi \xi}\left(\frac{\pi}{2} - \tan^{-1} \left(\frac{\xi}{k_0} \right) \right). 
\end{equation*}

The following result for spectral multipliers is required.
\begin{lemma} \label{lemmaspecmult}
Let $\Delta_\mu$ be the Laplace operator acting on a weighted manifold $M$ with smooth measure $\mu$. Let $F(\xi) = \pi/2 - \tan^{-1}(\xi)$. If $(M,\mu)$ satisfies the doubling condition and the heat kernel satisfies Gaussian upper bounds, then the operator $F(\sqrt{\Delta_\mu}/a)$, for any $a >0$, defined initially on $L^2(M,\mu)$ via the spectral theorem can be extended to a bounded operator on all $L^p(M,\mu)$ spaces and
\begin{equation*}
\left\| F \left(\frac{\sqrt{\Delta_\mu}}{a} \right) \right\|_{p \rightarrow p} \leq C_a < \infty,
\end{equation*}
for $1 \leq p \leq \infty$.
\end{lemma}

For the proof, see the sketch proof in Part I, which is in turn based on \cite{DOS}. 

\

We will also need the following boundedness results for one-dimensional kernels, acting on functions on a half-line with measure $r^{d-1} dr$, which we use as a simple model for radially symmetric operators on $\R^d$, or $\R^d \times \M_i$. 

\begin{lemma}\label{lem:d1d2} Let $d_1 \geq 1$ and $d_2 \geq 1$ be two real numbers. 
Consider the kernel $K(x,y)$, acting on functions defined on $[1, \infty)$, defined by
\begin{equation*}
K(x,y) = \begin{cases}
		x^{-a}y^{-b}, & \; 1 \leq x \leq y,\\
		x^{-a'}y^{-b'}, & \;  1 \leq y < x.
	\end{cases}
\end{equation*}
If $p(a+b-d_2) > d_1-d_2$, $p(a'+b'-d_2) > d_1-d_2$ and 
\begin{equation*}
\frac{d_1}{\min\{d_1,a' \}} < p < \frac{d_2}{\max\{0,d_2-b \}},
\end{equation*}
then $K$ is bounded as an operator from $L^p([1, \infty); r^{d_2 - 1} dr)$ to $L^p([1, \infty); r^{d_1 - 1} dr)$. 
\end{lemma}

\begin{proof}
It is sufficient to consider the case $x \leq y$, as the other follows from duality. 
\begin{eqnarray*}
\|K f\|_p^p &=& \int_1^{\infty} \left|\int_{x}^\infty x^{-a} y^{-b} f(y) \, y^{d_2-1} \dy \right|^p x^{d_1-1} \dx \\
&\leq& \int_1^{\infty} x^{-pa} \left(\int_{x}^\infty |f(y)|^p \, y^{d_2-1} \dy \right) \left(\int_{x}^\infty y^{-p'b} \, y^{d_2-1} \dy \right)^{p/p'} x^{d_1-1} \dx \\
&=& \int_1^{\infty} x^{-pa} \left(\int_{x}^\infty |f(y)|^p \, y^{d_2-1} \dy \right) \left(\left[\frac{y^{-p'b+d_2}}{-p'b+d_2} \right]_{x}^\infty \right)^{p/p'} x^{d_1-1} \dx \\
&\leq& C_1 \left(\int_1^{\infty} x^{-pa-pb+(p-1)d_2+d_1-1} \dx \right) \|f\|_p^p \\
&\leq& C_2 \|f\|_p^p.
\end{eqnarray*}
Note for the $y$-integral we used $-p'b+d_2 < 0$, which is equivalent to $p < d_2/\max\{0, d_2-b \}$. For $y < x$, we obtain by duality the condition $-pa'+d_1 < 0$, which is equivalent to $p > d_1/\min\{d_1,a' \}$. Furthermore, for the $x$-integrals to converge we require that $p(a+b-d_2) > d_1-d_2$ and $p(a'+b'-d_2) > d_1-d_2$.
\end{proof}

In the case that $d_1 = d_2 = d$, the lemma above requires that $a+b > d$ and $a'+b' > d$. The limiting case where $a+b = d$ and $a'+b' = d$ is covered by the following result. Due to scaling invariance we state the result on $(0, \infty)$ rather than $[1, \infty)$.

\begin{lemma}\label{lem:d} Let $d \geq 1$ be a real number. 
Consider the kernel $K(x,y)$, acting on functions defined on $(0, \infty)$, defined by
\begin{equation*}
K(x,y) = \begin{cases}
		x^{-a}y^{-b}, & \; 0 < x \leq y,\\
		x^{-a'}y^{-b'}, & \;  0 < y < x.
	\end{cases}
\end{equation*}
If $a+b = d = a' + b'$, then $K$ is bounded on $L^p((0, \infty); r^{d-1} dr)$ provided that $a' > a$ and 
\begin{equation}
\frac{d}{a'} < p < \frac{d}{a}. 
\label{p-aa'-cond}\end{equation}
\end{lemma}

\begin{proof}
The proof follows \cite[Theorem 5.1]{HS1}. We make the transformation $M$  
defined by 
$$
M : L^p((0, \infty); r^{d-1} dr) \to L^p((0, \infty); r^{-1} dr), \quad (Mf)(x) = x^{d/p} f(x).
$$
This is an isometry, and the corresponding operator $\tilde K$ has kernel given by 
$$
\tilde K(x,y) = x^{d/p} K(x,y) y^{d-d/p}.
$$
Under the conditions of the lemma, this is a function of $x/y$. 
We then change variable to $s = \log x$, inducing an isometry from $L^p((0, \infty); r^{-1} dr)$ to $L^p(\R; ds)$ and then $\tilde K$ becomes a convolution kernel $u(s-t)$, with 
$$
u(s) = \begin{cases} e^{(d/p-a) s}, \quad s \leq 0 \\
e^{(d/p-a') s}, \quad s > 0 .
\end{cases}
$$
The condition \eqref{p-aa'-cond} then ensures that this function $u$ is exponentially decaying in both directions, and is therefore in $L^1$. As is well-known, convolution with an $L^1$ function is a bounded map on all $L^p$ spaces on the real line. 
\end{proof}

\subsection{$L^p$ boundedness of the low energy Riesz transform}

\begin{proposition}
The Riesz transform localised to low energies, $\nabla F_<(\sqrt{\Delta})$, is of weak type $(1,1)$ and bounded on $L^p$ for $1 < p \leq 2$.
\end{proposition} 

\begin{proof}
We decompose the resolvent into several parts, $(\Delta+k^2)^{-1} = G_1(k) + G_2(k) + G_3(k) + G_4(k)  + G(k)S(k)$, and investigate the $L^p$-boundedness of each part. \\
\underline{$G_1$ term:} It has already been shown in Part I that the term supported on the $E_+$ end is bounded on $L^p$ for $1 < p < n_+$, and is of weak type $(1,1)$. So we consider the term supported on the $E_-$ end. This takes the form 
\begin{equation*}
\int_0^{k_0} \nabla \left((\Delta_{-} + k^2)^{-1}(z,z')\phi_-(z) \phi_-(z') \right) \dk
\end{equation*}
\begin{equation} \label{gradienthitsintegrals}
= \left[\int_0^{k_0} (\Delta_{-} + k^2)^{-1}(z,z') \nabla \phi_-(z) \phi_-(z')  \dk  + \int_0^{k_0} \nabla (\Delta_{-} + k^2)^{-1}(z,z')\phi_-(z) \phi_-(z')  \dk  \right].
\end{equation}
We treat the left integral first, where the gradient hits the $\phi(z)$ factor. Let $\mathcal{D}_r = \{(z,z') \in \M \times \M: \, d(z,z') \leq r \}$ and let $\chi_{\mathcal{D}_r}$ be the characteristic function for $\mathcal{D}_r$. We decompose the integral into near diagonal and away from diagonal parts.
For the near diagonal part, applying Schur test and using \eqref{resolvent-2} gives
\begin{equation*}
\left\| \int_0^{k_0} \chi_{\mathcal{D}_r}(\Delta_{-} + k^2)^{-1} \dk \right\|_{p \rightarrow p} \leq C_r.
\end{equation*}
for all $p \in [1, \infty]$. Now consider the resolvent kernel localized away from the diagonal on the $E_-$ end. For $s \in (1, \infty)$ we estimate the $L^{s'} \to L^\infty$ norm of the kernel: 
\begin{eqnarray*}
& & \left\|   (1-\chi_{\mathcal{D}_r})(\Delta_{-} + k^2)^{-1}(z,z') \right\|_{L^\infty(z); L^s(z')}  \\
&=& \sup_{z \in \M} \left(\int_{\M} |(1-\chi_{\mathcal{D}_r})(\Delta_{-} + k^2)^{-1}(z,z')|^s \dz' \right)^{1/s} \\
&=& \sup_{z \in \M}\left(\int_1^\infty (\log(kr)  e^{-ckr})^s  \; r \dr \right)^{1/s} \\
&=& \frac{1}{k^{2/s}} \sup_{z \in \M}\left(\int_k^\infty (\log(r) )^s  e^{-csr}  \; r \dr \right)^{1/s}\\
&\leq& Ck^{-2/s}.
\end{eqnarray*}
Similarly, we estimate the $L^1 \to L^s$ norm by 
\begin{eqnarray*}
\left\| (1-\chi_{\mathcal{D}_r})(\Delta_{-} + k^2)^{-1}(z,z') \right\|_{L^\infty(z'); L^s(z)} &\leq& Ck^{-2/s}.
\end{eqnarray*}
Applying the Schur test we see that $(1-\chi_{\mathcal{D}_r})(\Delta_{-} + k^2)^{-1}$ is bounded as an operator from $L^{s'} \rightarrow L^\infty$ and $L^1 \rightarrow L^s$ with operator norms bounded by $Ck^{-2/s}$. Now, given $p \in (1,2)$ and $q > p$, let $1 - \frac{1}{s} = \frac{1}{s'} = \frac{1}{p}- \frac{1}{q}$. Applying the Riesz-Thorin interpolation theorem gives
\begin{equation*}
\left\| (1-\chi_{\mathcal{D}_r})(\Delta_{-} + k^2)^{-1} \right\|_{p \rightarrow q} \leq Ck^{2(1/p-1/q)-2}.
\end{equation*}
Now using the compactness of the support of $\nabla \phi_-$, we obtain for $1/p - 1/q > 1/2$
\begin{eqnarray*}
\left\| \nabla \phi_-  \int_0^{k_0} (1-\chi_{\mathcal{D}_r}) (\Delta_-)^{-1} \dk \right\|_{p \rightarrow p} &\leq& C \left\|  \int_0^{k_0} (1-\chi_{\mathcal{D}_r}) (\Delta_{-} + k^2)^{-1} \dk \right\|_{p \rightarrow q} \\
&\leq& C \int_0^{k_0} k^{2(1/p-1/q)-2} \dk \\
&<& \infty.
\end{eqnarray*}
For any $p < 2$, we can choose $q$ large enough so that the condition $1/p - 1/q > 1/2$ holds. Combining the near diagonal and off diagonal norm estimates gives
\begin{equation*}
\left\| \nabla \phi_i  \int_0^{k_0} (\Delta_{-} + k^2)^{-1} \dk \right\|_{p \rightarrow p} < \infty,
\end{equation*}
for any $p < 2$. 

We now treat the right integral of $(\ref{gradienthitsintegrals})$, where the gradient hits the resolvent factor.
Then we can estimate 
\begin{eqnarray*}
\left\| \phi_i \int_0^{k_0} \nabla \left((\Delta_{-} + k^2)^{-1}  \right) \dk \, \phi_i \right\|_{p \rightarrow p} 
&=& \left\| \phi_i \nabla F_{<}(\sqrt{\Delta_{-}}) \, \phi_i \right\|_{p \rightarrow p} \\
&=& \left\| \phi_i \nabla (\Delta_{-})^{-1/2} \left(\frac{\pi}{2} - \tan^{-1}\left(\frac{\sqrt{\Delta_{-}}}{k_0} \right) \right)  \, \phi_i \right\|_{p \rightarrow p} \\
&\leq& C \left\| \nabla (\Delta_{-})^{-1/2} \right\|_{p \rightarrow p}.
\end{eqnarray*}
Note that we used Lemma \ref{lemmaspecmult} in the second last line. The norm in the last line is finite from standard results on Riesz transforms. We also obtain that weak type $(1,1)$ boundedness since $\pi/2-\tan^{-1}\left(\frac{\sqrt{\Delta_{-}}}{k_0}\right)$ is $L^1$ bounded.

\underline{$G_2$ term:} Note that $\nabla G_2$ is a pseudodifferential operator of order -1 which has a compactly supported Schwartz kernel depending smoothly on $k$. The integral must also be a pseudodifferential operator of order -1 which has a compactly supported Schwartz kernel. This means that this term is $L^p$-bounded for all $p \in [1,\infty]$.

\underline{$G_3$ term:}
We investigate the behaviours of
\begin{equation}
\sum_\pm \int_0^{k_0}  (\Delta_{\pm} + k^2)^{-1}(z_i^\circ,z') \phi_\pm(z') \nabla_z u_\pm(z,k) \dk
\label{G3-term}\end{equation}
where $z,z'$ lie on each of the ends. First consider the $-$ term, when both $z,z'$ are on the $E_-$ end. Then we get an estimate using \eqref{graduestimates} and \eqref{resolvent-2}
\begin{eqnarray*}
\int_0^{k_0} \big( r^{-2} + (\ilg k) r^{-1} \big) e^{-ckr} (1 + |\log (kr') |) e^{-ckr'} \, dk 
\end{eqnarray*}
For the $r^{-2}$ term, we can estimate the quantity above by 
$$
r^{-2} \int_0^{k_0}  (1 + |\log (kr') |) e^{-ckr'} \, dk \leq C r^{-2} {r'}^{-1}. 
$$
For the $(\ilg k) r^{-1}$ term, we note that $(\ilg k) |\log kr'| \leq 1$ since $k$ is small and $r' \geq R$ is large here. Therefore we can estimate this quantity by 
$$
r^{-1} \int_0^{k_0}  e^{-ckr} e^{-ckr'} \, dk = r^{-1} (r + r')^{-1} \sim \min(r^{-2}, r^{-1} {r'}^{-1}). 
$$
Now using Lemma~\ref{lem:d} with $d=2$, we see that this is bounded on $L^p$ for $1 < p < 2$. 

Next consider the $-$ term when $z$ is in $E_+ \cup K$ and $z'$ in $E_-$. We estimate the kernel, using \eqref{graduestimates2} instead of \eqref{graduestimates} by 
$$
 \ilg k \, \int_0^{k_0}r^{1-n_+} e^{-ckr} (1 + |\log (kr') |) e^{-ckr'} \, dk .
$$
This is estimated like the first case above, and we arrive at an estimate of 
$$
\min(r^{-n_+}, r^{1-n_+} {r'}^{-1}). 
$$
We apply Lemma~\ref{lem:d1d2} with $d_1 = n_+$, $d_2 = 2$ and $a = n_+-1$, $a' = n_+$, $a+b = a'+b' = n_+$. We see that this part of the operator is bounded on $L^p$ for $1 < p < 2$. 

Now we consider the $+$ term, when $z$ is in $E_- \cup K$ and $z'$ is in $E_+$. In this case we can estimate the kernel by 
\begin{eqnarray*}
\int_0^{k_0} \big( r^{-2} + (\ilg k) r^{-1} \big) e^{-ckr} {r'}^{2-n_+} e^{-ckr'} \, dk 
\end{eqnarray*}
Replacing $r^{-2} + (\ilg k) r^{-1}$ by $r^{-1}$, we estimate this quantity by 
$$
r^{-1} {r'}^{2-n_+} (r+r')^{-1} \sim \min ( r^{-2} r^{2-n_+}, r^{-1} {r'}^{1-n_+}). 
$$
Applying Lemma~\ref{lem:d1d2} with $d_1 = 2$, $d_2 = n_+$ and $a = 1$, $a' = 2$, $a + b = a' + b' = n_+$, we see that this part of the operator is bounded on $L^p$ for $1 < p < n_+$. 

Lastly, for the $+$ term when $z$ and $z'$ are both in $E_+$, we estimate \eqref{G3-term} by 
$$
\int_0^{k_0} r^{1-n_+} e^{-ckr} {r'}^{2-n_+} e^{-ckr'}  \, dk = C r^{1-n_+}{r'}^{2-n_+}(r+r')^{-1} \sim \min(r^{-n_+} {r'}^{2-n_+}, r^{1-n_+} {r'}^{1-n_+}). 
$$
We obtain boundedness from Lemma~\ref{lem:d1d2} with $d_1 = d_2 = n_+$, similarly to above.

\underline{$G_4$ term:} This term contributes the following to the low energy Riesz transform
\begin{equation*}
k_0 \sum_{i=1}^N \nabla \psi_i \langle \varphi_i, \cdot \rangle
\end{equation*}
where $\psi_i, \varphi_j \in C_c^\infty(\M)$. This operator is $L^p$-bounded for $p \in [1,\infty]$. 
\end{proof}

\underline{$GS$ term:} 
We have seen that the kernel $\big(\nabla G(k) S(k) \big)(z,z')$ is bounded pointwise by $C \nu(z,k) \mu'(z', k)$. This can be integrated in $k$ and estimated just as for the $G_3$ term above (in fact simpler as there are no $\log k$ factors that need to be cancelled with $\ilg k$ factors).

\subsection{$L^p$ boundedness of the high-energy Riesz transform}
This works just as in Part I. We have

\begin{proposition}[Proposition 5.1 of \cite{HS2}]\label{prop:he-Riesz}
The Riesz transform localized to high energies, $\nabla F_>(\sqrt{\Delta})$, is bounded on $L^p(\M)$ for $p$ in the range $(1, \infty)$. 
In addition, it is of weak-type $(1,1)$, that is, it is a bounded map from $L^1(\M)$ to $L^1_w(\M)$. 
\end{proposition}

The proof of this Proposition from \cite{HS2} works perfectly well in the present setting; there are no changes required to accommodate the two-dimensional Euclidean factor on the $E_-$ end. For the reader's convenience, we recall some ideas in the proof. 

The strategy is to decompose $F_>(\lambda)$ into two functions $F_>(\lambda) = G'_r(\lambda) + G''_r(\lambda)$, where $G'_r(\lambda)$ has Fourier transform supported in $[-r, r]$ and  $G''_r$ has Fourier transform supported outside the set $[-r/2, r/2]$. The support condition on the Fourier transform of $G'_r$ shows that $G'_r(\sqrt{\Delta})$ is supported in an $r$-neighbourhood of the diagonal. One also has that $G'_r(\lambda)$ is a symbol of order $-1$. From these two facts, one can show that $G'_r(\sqrt{\Delta})$ is a pseudodifferential operator of order $-1$, using existing results about symbolic functions of elliptic operators on \emph{compact} manifolds. We then find that $\nabla G'_r(\sqrt{\Delta})$ is a pseudodifferential operator of order zero, hence bounded on $L^p$ for $1 < p < \infty$ and of weak-type $(1,1)$. 

On the other hand, for $\nabla G''_r(\sqrt{\Delta})$, one shows that the kernel of this operator decays exponentially as the distance between the points tends to infinity. Given this, Schur's test shows that this operator is $L^p$ bounded for $1 \leq p \leq \infty$. This is achieved by passing from Schur-type estimates on the kernel to $L^2$ estimates. On $L^2$ one can relate $\nabla G''_r(\sqrt{\Delta})$ to $\sqrt{\Delta} G''_r(\sqrt{\Delta})$, and the latter operator is in the functional calculus of $\Delta$, so can be related directly to properties of the function $G''_r$. One also needs a fundamental Sobolev inequality of Cheeger-Gromov-Taylor \cite{CGT}, valid on all manifolds with $C^\infty$ bounded geometry, to pass from $L^2$ estimates back to supremum estimates. For the full details, see \cite[Section 5]{HS2}.


\section{Unboundedness of the Riesz transform for $p > 2$}

\begin{proposition}
The Riesz transform on $\M$ is not $L^p$-bounded for $p > 2$.
\end{proposition}

\begin{proof}
In the proofs of boundedness of the Riesz transform, we saw that the high energy Riesz transform is bounded for all $p \in (1, \infty)$,
and that many of the components of the low energy Riesz transform is all bounded at least for $p \in (1, n_+)$. It suffices to only consider the sum of those parts of the Riesz transform that were not shown to be bounded for some $p > 2$, and to show that sum is unbounded. 
Therefore we only consider the ``$\nabla \phi$'' part of the $-$ term of the $G_1$ component, and the $-$ term of the $G_3$ component. That is, it suffices to consider 
\begin{align}
\begin{split}
\int_0^{k_0} \nabla \phi_-(z)(\Delta_{-} + k^2)^{-1}(z,z') \phi_-(z')  + \nabla u_-(z,k) (\Delta_{-} + k^2)^{-1}(z_-^\circ,z') \phi_-(z') \dk
\end{split}
\end{align}
does not act boundedly on $L^p(\M)$ for $p > 2$. 

First observe that the $z$ variable in the resolvent factor can be replaced by $z_-^\circ$. Indeed, using the compact support of $\nabla \phi_-$, we apply the gradient estimate \eqref{gradresolvent-2} of the resolvent kernel to estimate
\begin{align}
\begin{split}
\int_0^{k_0} \nabla \phi_-(z)\left|(\Delta_{-} + k^2)^{-1}(z,z')  - (\Delta_{-} + k^2)^{-1}(z_-^\circ,z') \right| \phi_-(z') \dk \\
\leq |\nabla \phi_-(z)|  \int_0^{k_0}  \left( {r'}^{-1}  \exp \left(-kr' \right)\right) \dk \nonumber . 
\end{split}
\end{align}
This integral can be estimated by $O({r'}^{-2} )$, and is compactly supported in the left variable, and so is bounded for all $L^p$ for $p \in (1,\infty)$. So now we can combine the two terms. It suffices to  show that
\begin{equation*}
\int_0^{k_0} \left(\nabla \phi_-(z) + \nabla u_-(z,k) \right) (\Delta_{-} + k^2)^{-1}(z_-^\circ,z') \phi_-(z') \dk
\end{equation*}
is not $L^p$-bounded for $p > 2$. It suffices to consider only the $r$-component of the gradient.
Choose $\tau(z)$ to be smooth, nonnegative, compactly supported, not identically zero, and supported in the region $\{ r \geq r_0 \}$ on the end $E_-$, where $r_0$ is as in \eqref{res-lowerbound-2}.  It suffices to show 
\begin{equation*}
\tau(z) \int_0^{k_0} \partial_r \Big(\phi_-(z) + u_-(z,k) \Big) (\Delta_{-} + k^2)^{-1}(z^\circ,z') \phi_-(z') \dk
\end{equation*}
is not $L^p$-bounded for $p > 2$. 

Recall from Lemma \ref{u2lemma} that $\partial_r \left(u_-(z,k) + \phi_-(z)\right) \geq C \frac{\textup{ ilg} k}{r} 1_{kr \leq \epsilon} + O(kr^{-1})$,
for $r \geq r_0$. Then the $O(kr^{-1})$ term contributes a kernel that is bounded on $L^p$ for every $p > 2$. In fact, by \eqref{resolvent-2}, this term contributes a kernel that is dominated in magnitude by 
$$
 \int_0^{k_0} k r^{-1}  C \Big( 1 + | \log(k d(z,z')) | \exp(-ckd(z,z') ) \Big) \, dk \leq C r^{-1} {r'}^{-2}.
 $$
This kernel is $L^p$ in the left variable and $L^{p'}$ in the right variable for every $p > 2$, and hence is bounded on $L^p$ for every $p > 2$, as claimed. 

So it remains to treat the part of $\partial_r \left(u_-(z,k) + \phi_-(z)\right)$ which is bounded below by the nonnegative quantity $C \ilg k \, r^{-1} 1_{kr \leq \epsilon}$. We can now use the positivity of the resolvent, specifically the lower bound \eqref{res-lowerbound-2}. 
Set
\begin{equation*}
f(t) = \begin{cases}
\frac{\ilg t}{1+\ilg t} \qquad &\text{if } t < 1, \\
1 \qquad &\text{if } t \geq 1.
\end{cases}
\end{equation*}

Observe that for $kr' < 1$,
\begin{eqnarray*}
\ilg k &=& \frac{\ilg kr' \, \ilg \frac{1}{r'}}{\ilg \frac{1}{r'} + \ilg k r'} \\
&=& \ilg \frac{1}{r'} \left(\frac{\ilg kr'}{\ilg \frac{1}{r'} + \ilg kr'} \right) \\
&\geq& \ilg \frac{1}{r'} \left(\frac{\ilg kr'}{1 + \ilg kr'} \right) \\
&=& f(kr') \; \ilg \frac{1}{r'}.
\end{eqnarray*}
Now applying the lower bound \eqref{res-lowerbound-2} for the resolvent, and noting that $r'$ is equal to $d(z_-^\circ, z') + O(1)$, we obtain 
\begin{align*}
 \frac{\tau(z)}{r} &   \int_0^{k_0}  \ilg k  \, 1_{kr \leq \epsilon} \,  (\Delta_{-} + k^2)^{-1}(z_2^\circ,z')  \dk \\ 
 &\geq \frac{\tau(z)}{r} \ilg \frac{1}{r'} \int_0^{\epsilon r^{-1}} f(kr') c \Big( 1 + | \log(k r') | \exp(-Ckr' ) \Big) \dk \\
&\geq c \frac{\tau(z)}{r} \frac{\ilg \frac{1}{r'}}{r'} \int_0^{\epsilon} f(\kappa') \Big( 1 + | \log(\kappa') | \exp(-C\kappa' ) \Big)  \dkappa' \\
&\geq C \frac{\tau(z)}{r} \frac{\ilg \frac{1}{r'}}{r'}.
\end{align*}
The kernel in the last line is a rank one kernel, of the form $a(z) b(z')$, and is bounded on $L^p$ if and only if $a \in L^p$ and $b \in L^{p'}$. Since $\frac{\ilg \frac{1}{r'}}{r'} \in L^{p'}$ only for $p \leq 2$, this kernel fails to be $L^p$-bounded in the range $p>2$. 
Consequently the Riesz transform is also not bounded in this range. 
\end{proof}

\end{document}